\documentclass[12pt]{article}

\usepackage{latexsym}
\usepackage{indentfirst}
\usepackage{graphicx}
\usepackage{amsmath,amsfonts}
\usepackage{amssymb}
\usepackage{amsthm}
\usepackage{bm}
\usepackage{accents}
\usepackage{changepage}
\usepackage{color}
\usepackage{bigints}
\usepackage{enumerate}
\usepackage{theoremref}

\newtheorem{prop}{Proposition}[section]
\newtheorem*{defi*}{Definition}
\newtheorem{defi}[prop]{Definition}

\newtheorem{lem}[prop]{Lemma}
\newtheorem{rem}[prop]{Remark}
\newtheorem{thm}[prop]{Theorem}
\newtheorem{coro}[prop]{Corollary}

\newtheorem*{thm0}{Theorem 0}
\newtheorem*{thmI}{Theorem I}
\newtheorem*{thmII}{Theorem II}

\numberwithin{equation}{section}

\begin{document}

\noindent
{\Large\bf A natural extension of Markov processes and applications to singular SDEs
}\\

\noindent
Lucian Beznea\footnote{Simion Stoilow Institute of Mathematics  of the Romanian Academy,
 Research unit No. 2, 
P.O. Box \mbox{1-764,} RO-014700 Bucharest, Romania,  University of Bucharest, 
Faculty of Mathematics and Computer Science, and Centre Francophone en Math\'ematique de Bucarest
(e-mail: lucian.beznea@imar.ro)},
Iulian C\^{i}mpean\footnote{Simion Stoilow Institute of Mathematics  of the Romanian Academy,
Research unit No. 2, P.O. Box  1-764, RO-014700 Bucharest, Romania
(e-mail: iulian.cimpean@imar.ro)},
Michael R\"ockner\footnote{Fakult\"at f\"ur Mathematik, Universit\"at Bielefeld,
Postfach 100 131, D-33501 Bielefeld, Germany, and Academy for Mathematics and Systems Science, CAS, Beijing 
(e-mail: roeckner@mathematik.uni-bielefeld.de)}
\\[5mm]

\noindent
{\bf Abstract.} 
We develop a general method for extending Markov processes to a larger state space such that the added points form a polar set.
The so obtained extension is an improvement on the standard trivial extension in which case the process is made stuck in the added points, and it renders a new technique of constructing extended solutions to S(P)DEs from all starting points, in such a way that they are solutions at least after any strictly positive time. 
Concretely, we adopt this strategy to study SDEs with singular coefficients on an infinite dimensional state space (e.g. SPDEs of evolutionary type), for which one often encounters the situation where not every point in the space is allowed as an initial condition. 
The same can happen when constructing solutions of martingale problems or Markov processes from (generalized) Dirichlet forms, to which our new technique also applies.
\\[2mm]

\\



\section{Introduction and the main results}
It is a common phenomenon in the construction of Markov processes in infinite dimensions, say on a separable $\mathbb{R}$-Hilbert space $H$ (be it through solving a martingale problem or an SDE on $H$, e.g. an SPDE of evolution type) that one has to restrict the set of "allowed" starting points.
This in part is due to the fact of non-existence of fundamental solutions (heat kernels) for the corresponding generating Kolmogorov operators or the singular behavior of the heat kernel when $t\to 0$, if it exists.
Such a situation one encounters, in particular, when constructing Markov processes starting from (generalized) Dirichlet forms (see Subsection 1.3 below), but also when one tries to construct solutions to SDEs, e.g. stochastic reaction diffusion equations.
In fact, even knowing in advance that the corresponding resolvent is (Lipschitz) strong Feller or moreover with corresponding transition semigroup being (Lipschitz) strong Feller, such a situation can occur.
A surprising fact is that in the latter case there is a simple counter example on $\mathbb{R}\setminus \{0\}$ where the strong Feller property of the transition semigroup does not guarantee that one can solve the corresponding martingale problem for any starting point (see Corollary \ref{rem 3.16} from Appendix).

The general aim of this work is to develop a natural way of extending the state space of a Markov process which at a first stage is constructed on a smaller space (see Theorem \ref{thm 2.2}) of "good"points, such that when the extended process starts from the added "bad" points, it will immediately enter the space of "good" points, where it stays for the rest of the time, i.e., in potential theoretical terms, the set of bad starting points is polar.
We call this procedure "natural extension", and the main tools to develop it are potential theoretic, based on a Ray type completion of the state space of a right Markov process as developed in \cite{BeRo11a} and \cite{BeBo04a}.
We want to stress  that our natural extension is completely different from the usual one, called trivial extension, where the Markov process is made stuck for bad starting points.
In the latter case, the process has no relation to the corresponding SDE or martingale problem when started at such a bad point, while in our case the process immediately enters the set of good points (see the beginning of Subsection 1.3 for more details).

Examples in infinite dimensions where bad starting points occur are discussed in  \cite{DaRo02a}, \cite{DaRo02b}, \cite{DaRoWa09}, \cite{DaFlPriRo13}, \cite{DaFlPriRo13}, \cite{DaFlRoVe16}.
More precisely, the solutions to the SDEs therein considered were constructed for all starting points except the ones of the abstract (though negligible) set of 'bad' starting points.
To avoid any confusion, we would like to point out that although the main results from \cite{DaRo02a} claim the existence of the associated diffusion process from all starting points, there is a small gap which was treated afterward in \cite{DaRo02b}.
In this regard, the question of existence (and uniqueness) of solutions which are allowed to start from such bad points remained open, and one concrete goal of this work is to give a positive answer to this question, based on the general technique of extending Markov processes developed in Subsection 1.2. 
More precisely, in \thnameref{thm 3.4} and \thnameref{thm 3.16} we show that the solutions obtained in the aforementioned papers can be extended to Lipschitz strong Feller diffusions on the entire space of starting points, in a unique natural way, so that the associated martingale problems can be solved for all starting points. 
We thus show that all statements made in \cite{DaRo02a} concerning the solutions of the martingale problem for any starting point are indeed correct and the completion of the proof is contained in this paper.
Moreover, we show that when the extended diffusions start from a bad point, they become solutions in the classical sense for the corresponding SDEs, after any strictly positive moment of time.
It seems that in most general situations this is the best result one could possibly expect. 

Concerning other previously known extension techniques for SDEs, let us mention the method of "generalized solutions" from \cite{DaZa14}, Subsection 7.2.4. (see also \cite{Ge14}), where a pathwise extension is constructed for the solution of an SDE with continuous and dissipative drift, which is constructed at a first stage only on a smaller space. 
Even in this situation, when starting from a bad point, the so extended process is not always associated to a corresponding SDE.
In this paper, when we apply our natural extension to an SDE, the drift consists not only of a singular dissipative part, but also a merely bounded part, so that the above extension can not be applied; though, we make use of such an extension when the merely bounded part is zero, see Subsection 1.1. 
Moreover, our natural extension enjoys a smoothing-type effect, in the sense that when it starts from a bad point, it will immediately enter the set of good points, where it remains for the rest of time. 
Due to this behaviour, we are able to show that the so extended Markov process solves the corresponding SDE after any infinitesimally small time $t>0$.

Concerning the structure of this paper, we would like to mention that instead of starting with the general result on natural extensions of Markov processes and then look at the applications to S(P)DEs, we preferred to do it the other way around, with the hope that the reader would naturally be led from concrete difficulties arising from singular SDEs, to the importance of considering the general problem of extending the state space of a Markov process.
Concretely, the remainder of this section is structured in three subsections: 
In Subsection 1.1 we recall the SDEs under consideration together with several known results which are needed in this paper; then we present our new results (\thnameref{thm 3.4} and \thnameref{thm 3.16}), which in particular solve the left over problems of the 'bad' starting points in \cite{DaRo02a, DaRo02b}.
Subsection 1.2 is devoted to the general result (Theorem \ref{thm 2.2}) concerning the extension of the state space of a Markov process, where most of the potential theoretic techniques occur. 
We emphasize that the main result of this subsection is the key instrument used to prove the results of the first subsection. 
In Subsection 1.3, as another application of the general results from Subsection 1.2, we apply our technique of natural extension to construct right processes from (generalized) Dirichlet forms, so that they can start in a natural way from all points of the state space, thus avoiding the trivial modification which is usually implemented (see Corollary \ref{coro 3.1}).

Section 2 contains the proofs of the results stated in Section 1, and it is again organized in three subsections, corresponding to those from Section 1.

Finally, in the Appendix we give an overview of right processes and their potential theory, where one of the aims is to explain carefully the role of different topologies like the fine topology and its natural topologies, which are frequently encountered within the main body of the paper. 
Therefore, our recommendation to the reader who is particularly interested in Subsection 1.2 and the details of the proofs in Section 2, is to start with the Appendix.

\subsection{Stochastic equations in Hilbert spaces with nonregular drifts} 

We place ourselves into the framework of the papers \cite{DaRo02a}, \cite{DaRoWa09}, and \cite{DaFlRoVe16}.
More precisely, let $(H,\langle\cdot,\cdot\rangle)$ be a real separable Hilbert space (with norm $|\cdot|$) and consider the stochastic differential equation 

\begin{equation} \label{eq 3.1}
  \begin{cases}
     dX(t)=(AX(t)+F_0(X(t))+B(X(t)))dt+\sigma dW(t)\\
    X(0)=x\in H,
  \end{cases}
\end{equation}
where $W$ is an $H$-valued cylindrical Wiener process on some probability space.

Concerning the coefficients appearing in (\ref{eq 3.1}), the following two hypotheses will be in force for the rest of the paper.

\vspace{0.2cm}
\noindent
{\bf Hypothesis 1.} 
\begin{enumerate}[(i)]
\item $A:D(A)\subset H \rightarrow H$ is a self-adjoint linear operator which generates a $C_0$-semigroup $T_t=e^{tA}$ on $H$, and there exists $\omega \in \mathbb{R}$ such that
$$
\langle Ax, x \rangle \leq \omega |x|^{2} \quad \mbox{   for all } x\in D(A).
$$
\item $\sigma$ is symmetric and positive definite such that $\sigma^{-1} \in L(H)$ (for simplicity one may assume that $\sigma=Id$) and for some $\alpha >0$
$$
\int_0^{\infty}(1+t^{-\alpha})|T_t  |_{\rm HS}^{2} \; dt < \infty,
$$
where $|\cdot|_{\rm HS}$ denotes the Hilbert-Schmidt norm.
\item $F_0$ is a (possibly) nonlinear mapping given by 
$$
F_0(x):=\mathop{arg \, min}\limits_{y\in F(x)}|y|, \quad x\in D(F),
$$
where $F:D(F)\subset H \rightarrow 2^{H}$ is an $m$-dissipative mapping, i.e. 
$$
\langle u-v,\; x-y\rangle \leq 0 \; \mbox{ for all } x,y\in D(F), \; u\in F(x),\; v\in F(y),
$$
and ${\rm Range}\;({\sf I}-F):= \mathop{\bigcup}\limits_{x\in D(F)}(x-F(x))=H$.
\end{enumerate}

The {\it Kolmogorov operator} associated to \eqref{eq 3.1} with $B=0$ is
$$
L_0 \varphi(x) =\frac{1}{2} {\rm Tr} [\sigma^{2}D^{2}\varphi(x)]+ \langle x , A D\varphi(x)\rangle + \langle F_0(x) , D\varphi(x) \rangle, \quad x\in D(F), \varphi \in \mathcal{E}_A(H),
$$
where $\mathcal{E}_A(H)$ is the linear space generated by the (real parts of) functions of type $\varphi(x)=\exp\{i\langle x, h \rangle\}$ with $h \in D(A)$.

\vspace{0.2cm}
\noindent
{\bf Hypothesis 2.} There exists a Borel probability measure $\nu$ on $H$ such that
\begin{enumerate}[(i)]
\item $ \displaystyle\int_{D(F)}(1+|x|^{4})(1+|F_0(x)|^{2}) \; \nu(dx) <\infty$.
\item $\displaystyle\int_H L_0 \varphi \; d\nu =0 \quad \mbox{ for all } \varphi \in \mathcal{E}_A(H).$
\item $\nu(D(F))=1$.
\end{enumerate}

For an exposition of concrete examples when the previous two hypotheses are fulfilled, we refer to \cite{DaRo02a}, \cite{DaRoWa09}, and \cite{DaFlRoVe16}.

Let $H_0:= supp(\nu)$ and $Lip_b(H_0)$ denote the space of all bounded Lipschitz functions on $H_0$.
By $b\mathcal{B}(H_0)$ we denote the space of all bounded and measurable functions from $H_0$ to $\mathbb{R}$; this notation will be later used for other spaces instead of $H_0$, with the same meaning.

We summarize now some of the main results from \cite{DaRo02a} and \cite{DaRoWa09} which we particularly rely on, more precisely (parts of) Theorem 2.3, Proposition 5.2, corollaries 5.3 and 5.4, and respectively Theorem 1.6 and Corollary 1.7.

\begin{thm0} \thlabel{thm 3.2} The following assertions hold.
\begin{enumerate}[(i)]
\item (cf. [Da Prato\slash R: PTRF 2002]) $(L_0, \mathcal{E}_A(H))$ is closable on $L^{2}(H,\nu)$, its closure denoted by $(L, D(L))$ is $m$-dissipative and:

\quad (i.1) there exists a Lipschitz strong Feller Markovian semigroup of kernels on $H_0$ denoted by $(P_t)_{t \geq 0}$ such that $\mathop{\lim}\limits_{t\to 0}P_tf=f$ pointwise on $H_0$ for all $f\in Lip_b(H_0)$; by (Lipschitz) strong Feller we mean $P_t(b\mathcal{B}(H_0))\subset C_b(H_0) (\mbox{resp. } Lip_b(H_0))$.

\quad (i.2) $\nu$ is invariant for $(P_t)_{t \geq 0}$ and the extension of $(P_t)_{t\geq 0}$ to $L^{2}(\nu)$ is the strongly continuous semigroup generated by $L$.
\item (cf. [Da Prato\slash R \slash Wang: JFA 2009]) $\nu$ satisfying Hypothesis 2 is unique, $P_t(L^{q}(H,\nu))\subset C(H_0)$, and the following Harnack inequality holds
\begin{equation*} 
(P_tf(x))^{q}\leq P_tf^{q}(y)e^{|\sigma^{-1}|^2\frac{p\omega |x-y|^{2}}{(q-1)(1-e^{-2\omega t})}}
\end{equation*}
for all $f\geq 0, t>0, q\in (1,\infty), x,y\in H_0$.
In particular, $P_t(dx)<<\nu,\; t>0.$
\end{enumerate}
\end{thm0}

\begin{rem} \label{rem 3.3}
\begin{enumerate}[(i)]
\item Hypothesis 1, (ii) implies that $\rm{tr}(A^{-1})<\infty$, and because Hypothesis 2 is in force, we can apply \cite{BoDaRo96}, Theorem 1.1 to deduce that $\nu << N(0, \frac{1}{2}A^{-1})$. This will be useful later to prove It\^o's formula for Lipschitz functions; see Proposition \ref{prop 3.8}.
\item Let $\mathcal{F}\mathcal{C}_b^{2}$ denote the space of functions of type $\phi(\varphi_1,\dots,\varphi_k)$ for all $k\geq 1$ and bounded functions $\phi:\mathbb{R}^{k}\rightarrow \mathbb{R}$ with continuous and bounded first and second derivatives, where $\varphi_k$ are defined in the proof of \thnameref{thm 3.4}.
Then by Proposition 3.3 from \cite{DaFlRoVe16}, we have that $\mathcal{F}\mathcal{C}_b^{2}$ is a core for $(L, D(L))$.
\end{enumerate}
\end{rem}

Let us take a moment to explain briefly the strategy used in the above mentioned papers in order to construct solutions for equation (\ref{eq 3.1}), so that we will have a clear context which leads us to the main goal of the present work.
Consider first the case $B\equiv 0$, as in \cite{DaRo02a} and \cite{DaRoWa09}.
The idea is to show that the operator $(L,D(L))$ fits in the framework of \cite{St99b} (or more particularly \cite{St99a}), i.e. it is the generator of a quasi-regular local generalized Dirichlet form, so that there exist a set $M\in \mathcal{B}(H_0)$ s.t. $\nu(H_0\setminus M)=0$ and a conservative normal strong Markov process with continuous paths on $M$ whose transition function is precisely the restriction to $M$ of $(P_t)_{t\geq 0}$.
This process is then shown to satisfy the martingale problem for the canonical projections on the directions given by an orthonormal basis which diagonalize $A-\omega{\rm Id}$ for some $\omega >0$. 
In fact, it is shown that the corresponding martingales are standard real valued Brownian motions, which means that the constructed Markov diffusion is a weak solution for (\ref{eq 3.1}), case $B\equiv 0$. Then, based on the Yamada-Watanabe type results from \cite{On04}, pathwise uniqueness and hence the existence of strong solutions are obtained. The strong Feller properties and Wang's Harnack inequalities are obtained by an approximation technique which regularize $F_0$ by convolution with infinite dimensional Gaussian semigroups.
The case when $B$ is bounded and measurable is treated in \cite{DaFlRoVe16}, but $F_0$ is restricted to be the sub-gradient of a convex function.
In this case, the existence of a solution is ensured by a Girsanov transformation performed on the solution of (\ref{eq 3.1}) for $B\equiv 0$, but again, we stress that this is possible only on the smaller set $M$.
Pathwise uniqueness is then obtained by an infinite dimensional Zvonkin-type transformation.

\vspace{0.2 cm}
\noindent
{\bf Main goal.} Our central concern in this paper is to deal with the 'bad' starting points from $H_0\setminus M$. 
More precisely, the aim is the following: first, show that there exists a Lipschitz strong Feller diffusion Markov process on the entire space $H_0$ with transition function $(P_t)_{t\geq 0}$, which solves the asociated martingale problem for all starting points $x\in H_0$; second, investigate if the extended Markov process starting from a bad 'point' remains a classical solution for the SDE (\ref{eq 3.1}).

We will split the study in two, the case when $B\equiv 0$ and the case when $B$ is bounded and measurable, but before that, let us point out that the first part of the main goal can be extracted and treated as a particular case of a general extension problem of the state space of a Markov process, which is in fact of main interest:

\vspace{0.2 cm}
\noindent
{\bf A general extension problem.}
Let $\mathcal{U}:=(U_\alpha)_{\alpha>0}$ be the Markovian resolvent of kernels associated with $(P_t)_{t\geq 0}$, 
$$
U_\alpha f:=\int_0^{\infty}e^{-\alpha t} P_tfdt \mbox{ for all } f\in b\mathcal{B}(H_0).
$$
Forgetting that we deal with SDEs on Hilbert spaces but keeping in mind that $M$ plays the role of the set of 'good' starting points, we place ourselves in the following abstract situation: we are given a Markovian resolvent of kernels $\mathcal{U}$ on a topological space $E$ (replacing $H_0$ above), and a subset $M\subset \mathcal{B}(E)$ for which $U_\alpha (1_{E\setminus M})\equiv 0, \alpha>0$, so that the restriction of $\mathcal{U}$ from $E$ to $M$ is the resolvent of a normal strong Markov process with right continuous (or continuous) paths on $M$. 
Is it possible to extend this process to the entire space $E$ so that it has resolvent $\mathcal{U}$, and hopefully having the property that if it starts from a point $x\in E\setminus M$, it will immediately enter $M$? 
We treat this abstract problem separately in Subsection 1.2, with the emphasis that the general result obtained there are used to prove the main results concerning equation (\ref{eq 3.1}), which we state in the sequel.
For the forthcoming potential theoretical notions (like right process or polar set) we refer to the Appendix.

\begin{thmI} \thlabel{thm 3.4}
Assume that $B\equiv 0$ and keep all the notations from \thnameref{thm 3.2}. Then the following assertions hold.

\vspace{0.2cm}
\noindent
(i) There exists a set $M \subset H_0$ such that $H_0\setminus M$ is polar and for each $x\in M$ there exists a pathwise unique continuous strong solution $(X(t,x))_{t\geq 0}$ (in the mild sense) to \eqref{eq 3.1} starting from $x$.
Moreover, if $x\in H_0\setminus M$ then there exists a generalized solution $(X(t,x))_{t\geq 0}$ starting from $x$, in the sense of [Da Prato\slash Zabczyk 2014]. 

\vspace{0.2cm}
\noindent
(ii) There exists a conservative right (strong) Markov process 
$X=(\Omega, \mathcal{F}, (\mathcal{F}_t)_{t\geq 0}, (X(t))_{t\geq 0}, \newline(\theta(t))_{t\geq 0}, (\mathbb{P}^x)_{x\in H_0})$ on $H_0$ (see Definition \ref{defi 4.4} below)
with $|\cdot|$-continuous paths and transition semigroup $(P_t)_{t \geq 0}$.
In particular, 
$$\mathbb{P}^{x}\circ X(\cdot)^{-1}=\mathbb{P}\circ X(\cdot,x)^{-1}\quad \mbox{ for all } x\in H_0.
$$
In addition, the following assertions hold:

\vspace{0.2 cm}
(ii.1) For all $x\in H_0$ we have $\mathbb{P}^{x}(X(t) \in M \mbox{ for all } t>0)=1$, where $M$ is the set from (i).

\vspace{0.2 cm}
(ii.2)  For every $x\in H_0,$ $\mathbb{P}_x$ solves the martingale problem 
for $L$ with test function space 
%
%
\begin{displaymath}
D_0:=\left\{ \varphi \in D(L)\cap C_{b}(H)|\; L\varphi \in 
L^{\infty}(H,\nu) \right\}
\end{displaymath}
and initial condition $x,$ i.e. $\mathbb{P}_x$-a.s. $X(0)=x$ and 
\begin{equation*}
\varphi (X(t))-\varphi (X(0))-\int_{0}^{t}L\varphi (X(s))ds,\;\; t\ge 0,
\end{equation*}
is a continuous $({\cal F}_t)$-martingale for all $\varphi \in D_{0}$.

\vspace{0.2 cm}
(ii.3) If $x\in H_0\setminus M$ and $\varepsilon >0$ is fixed, then under $\mathbb{P}^{x}$ it holds that $(X(t+\varepsilon))_{t\geq 0}$ is a probabilistically weak solution to \eqref{eq 3.1} (in the mild sense) starting from $X(\varepsilon)$.

\vspace{0.2 cm}
\noindent
(iii)  If $x\in H_0\setminus M$ and $\varepsilon >0$ is fixed, then equation \eqref{eq 3.1} has a pathwise unique continuous strong solution with initial distribution $\mathbb{P}^{x}\circ X(\varepsilon)^{-1}$.
\end{thmI}

\begin{rem}
\begin{enumerate}[(i)]
\item Obviously, since $X$ is a Markov process with transition semigroup $(P_t)_{t\geq 0}$, the laws $\mathbb{P}^{x}\circ X^{-1},\; x\in H_0$, are uniquely determined by these two properties.
\item In \thnameref{thm 3.4}, (i), the existence of a normal Markov process on the entire space $H_0$, with $|\cdot|$-continuous paths and transition function $(P_t)_{t\geq 0}$, follows from the existence of a generalized solution in the sense of \cite{DaZa14}, without making use of the general extension results from Section 1.2 (Theorem \ref{thm 2.2} or Corollary \ref{coro 3.1}). 
On the other hand, the fact that the set $E\setminus M$ is never hit by the so obtained process does not directly follow knowing that the latter is merely a generalized solution, and we stress that this property is crucial to prove \thnameref{thm 3.4}, (iii). 
Instead, the fact that $E\setminus M$ is indeed polar is a genuine product of Theorem \ref{thm 2.2}.

The forthcoming case $B\not\equiv 0$ is completely different because generalized solutions are no longer available, and the extension results obtained in Subsection 1.2 are employed in a crucial way even for the construction of the Markov process so that it can start from all points in $H_0$.
\end{enumerate}
\end{rem}

\noindent
{\bf The case when $B$ is bounded.} We keep the same notations as before. 
In order to study equation (\ref{eq 3.1}) when $B$ is bounded, the strategy is to use the Girsanov transformation for all starting points $x\in H_0$, which is not straightforward at all for $x\in H_0\setminus M$ (see Remarks \ref{rem 3.6} and \ref{rem 3.7} below). 
It turns out that in order to handle the 'bad' starting points, it is more suitable to perform the Girsanov transformation on the generalized solution of (\ref{eq 3.1}), with $B\equiv 0$, instead of the right process $X$ given by \thnameref{thm 3.4}, (ii).
So let us fix a cylindrical Wiener process $\widetilde{W}$ on a stochastic basis $(\widetilde{\Omega}, \widetilde{\mathcal{F}}, (\widetilde{\mathcal{F}})_t, \widetilde{\mathbb{P}})$, and take $(X(t,x))_{t\geq 0}$ to be the {\it generalized solution} given by \thnameref{thm 3.4}.
For each $t>0$, we define the Markov kernels
\begin{equation} \label{eq 3.7}
Q_tf(x):=\mathbb{E}^{\widetilde{\mathbb{P}}}\{f(X(t,x)) \rho_t^{x}\}
\end{equation}
for all $f \in b\mathcal{B}(H_0)$ and $x\in H_0$, where
\begin{equation} \label{eq 3.8}
\rho_t^{x}:=e^{\int_0^{t}\langle B(X(s,x))d\widetilde{W}(s) \rangle-\frac{1}{2}\int_0^{t}|B|^{2}(X(s,x))ds}
\end{equation}
are continuous $\widetilde{\mathcal{F}}_t$-martingales by Novikov's condition.
As expected, it turns out that $(Q_t)_{t\geq 0}$ has the semigroup property, but we draw the attention that because $B$ is not continuous, the proof is more delicate; see Proposition \ref{prop 2.4} below.

\begin{rem} \label{rem 3.6}
If $x\in M$, then by an infinite dimensional Girsanov transform (see \cite{LiRo15}, Appendix I) we get that $(X(t,x))_{t\in [0,T]}$ is a solution for equation (\ref{eq 3.1}) starting at $x$ under $d\mathbb{Q}_T^{x}:=\rho_T^{x}\;d\mathbb{P}$, which is unique in law.
However, this transformation can not be applied in the standard way if $x\in H_0 \setminus M$, because we don't know if $(X(t,x))_{t\in [0,T]}$ is a classical solution for (\ref{eq 3.1}), with $B\equiv 0$; however, it becomes a solution after each time $\varepsilon>0$, as in \thnameref{thm 3.4}, (ii.3), so we could apply Girsanov transform to obtain solutions for (\ref{eq 3.1}) after each time $\varepsilon>0$ under some new probabilities $\mathbb{Q}_T^{x,\varepsilon}$; but it is not clear how $(\mathbb{Q}_T^{x,\varepsilon})_{T,\varepsilon}$ can be superposed to obtain a global probability $\mathbb{Q}^{x}$ under which $(X(t,x))_{t\geq 0}$ becomes a continuous normal Markov solution for (\ref{eq 3.1}).
Instead, the kernels $Q_t(\cdot,x)$ given by (\ref{eq 3.7}) make sense for all $x\in H_0$, and our aim is to show that there exists a continuous normal strong Markov process $Y$ on $H_0$, with transition function $(Q_t)_{t\geq 0}$ and Lipschitz strong Feller resolvent.
The strategy adopted in Subsection 2.3 is to make use of Theorem \ref{thm 2.2} (or it's corollaries) and a resolvent formula (see \thnameref{thm 3.4} below), in order to obtain concurrently both the existence and the Lipschitz strong Feller property. 
\end{rem}

Recall that by [DaPFlRoVe 16], Lemma 3.7, if $\alpha\geq 4\pi|B|_\infty^{2}$, then both
\begin{equation}\label{eq 3.17}
\langle B, \nabla U_\alpha \rangle, (I-\langle B, \nabla U_\alpha \rangle)^{-1}:L^{\infty}(H,\nu)\rightarrow L^{\infty}(H,\nu)
\end{equation}
are well defined bounded operators with norms less then $2$, and 
\begin{equation} \label{estimate}
|U_\alpha(I-\langle B, \nabla U_\alpha \rangle)^{-1}f|_{\rm Lip}\leq 2\sqrt{\frac{\pi}{\alpha}}|f|_\infty.
\end{equation}

Further, we denote by $\mathcal{V}:=(V_\alpha)_{\alpha>0}$ the resolvent of kernels associated to $(Q_t)_{t\geq 0}$, i.e. for $\alpha>0$ and $f\in b\mathcal{B}(H_0)$
\begin{equation*}
V_\alpha f(x):=\int_0^{\infty}e^{-\alpha t} Q_tf(x) dt, \quad x\in H_0.
\end{equation*}

Our last main result heavily relies on the following relation between the resolvents $\mathcal{U}$ and $\mathcal{V}$, which could be itself of interest.
We emphasize that (\ref{eq 3.18}) is not hard to prove $\nu$-a.s. by an operatorial approach, as it was done in \cite{DaFlRoVe16}, Proposition 3.8; however, we need it pointwise on the entire $H_0$, and to do this we had to come up with a completely different proof, based on It\^o's formula for Lipschitz functions obtained in Proposition \ref{prop 3.8}.

\begin{thm} \label{thm 3.13}
If $\alpha\geq 4\pi|B|_\infty^{2}$ and $f\in b\mathcal{B}(H_0)$, then
\begin{equation}\label{eq 3.18}
V_\alpha f=U_\alpha(I-\langle B, \nabla U_\alpha \rangle)^{-1}f.
\end{equation}
In particular, $|V_\alpha f|_{\rm Lip}\leq 2\sqrt{\frac{\pi}{\alpha}}|f|_\infty$ and $\mathcal{V}$ is Lipschitz strong Feller.
\end{thm}

Since $B\not\equiv 0$, the Kolmogorov operator associated to (\ref{eq 3.1}) is now
$$
L_0^{B} \varphi =L_0 \varphi + \langle B , D\varphi \rangle, \quad \varphi \in \mathcal{E}_A(H).$$

We can conclude now:

\begin{thmII} \thlabel{thm 3.16}
There exists a conservative right Markov process $Y=(\Omega, \mathcal{G}, (\mathcal{G}_t)_{t\geq 0}, (Y(t))_{t\geq 0}, \newline(\theta(t))_{t\geq 0}, (\mathbb{Q}^x)_{x\in H_0})$ on $H_0$ with a.s. $|\cdot|$-continuous paths, transition function $(Q_t)_{t \geq 0}$, and Lipschitz strong Feller resolvent $\mathcal{V}$.
In addition, the following assertions hold:
\begin{enumerate}[(i)]
\item $(Q_t)_{t\geq 0}$ extends to a strongly continuous semigroup on $L^{2}(\nu)$, whose infinitesimal generator $(L^{B},D(L^{B}))$ is the closure of $(L_0^{B}, \mathcal{E}_A(H))$; in particular, $D(L^{B})=D(L)$.
\item For every $x\in H_0,$ $\mathbb{Q}^x$ solves the martingale problem 
for $L^{B}$ with the same test function space as in \thnameref{thm 3.4}
and initial condition $x$, i.e. $Y(0)=x$ $\mathbb{Q}^x$-a.s. and under $\mathbb{Q}^x$ 
\begin{equation} \label{eq 3.22}
\varphi (Y(t))-\varphi (Y(0))-\int_{0}^{t}L^{B}\varphi (Y(s))ds,\;\; t\ge 0,
\end{equation}
is a continuous $({\cal F}_t)$-martingale for all $\varphi \in D_{0}$.
\item If $x\in M$, then under $\mathbb{Q}^{x}$, the Markov process $Y$ is a (unique in law) probabilistically weak solution for equation (\ref{eq 3.1}) (in the mild sense), which remains in $M$.
\item If $x\in H_0\setminus M$ and $\varepsilon >0$ is arbitrarily fixed, then under $\mathbb{Q}^{x}$ we have that $(Y(t+\varepsilon))_{t\geq 0}$ is a solution to equation (\ref{eq 3.1}) (in the mild sense) starting from $Y(\varepsilon)$ and remaining in $M$.
\end{enumerate}
\end{thmII}

\begin{rem}
Since on $M$ the process $Y$ is a weak solution for (\ref{eq 3.1}), case $B\not\equiv 0$, one can easily see that It\^o's formula from Proposition \ref{prop 3.8} remains valid for $L$ replaced by $L^{B}$.
\end{rem}

\vspace{0.4 cm}
\noindent
{\bf A concrete example: reaction-diffusion equation.}
First of all, let us mention that since \thnameref{thm 3.4} and II hold whenever Hypotheses 1 and 2 are fulfilled, they apply to all exemples considered in \cite{DaRo02a}, \cite{DaRoWa09}, and \cite{DaFlRoVe16}.
In this paragraph we look at such a concrete example, and in addition, we provide explicit descriptions of $H_0$ and the set of "good" starting points $M$, so that we can fully profit from the polarity of $H_0\setminus M$. 
 
Let $H=L^{2}(0,1)$ and define the operator $A$ by
\begin{equation*}
A=\Delta, \quad D(A)=H^{2}(0,1) \cap H_0^{1}(0,1).
\end{equation*}
Also, for fixed $m\geq 1$, consider the convex functional $V:H\rightarrow (\infty, +\infty]$ given by

$$
 V(x):=\left\{
\begin{aligned}
  &|x|_{L^{m+1}(0,1) }^{m+1} &\mbox{ if } x \in L^{m+1}(0,1)\\
  & \infty  &\mbox{otherwise}
\end{aligned}
\right. 
.
$$
Then $F: D(F)\subset H \rightarrow H$ given by
\begin{equation*}
F(x)=-\nabla V(x)=-(m+1)x|x|^{m-1} \mbox{ for } x\in D(F):= L^{2m}(0,1)
\end{equation*}
 is $m$-dissipative.

Assume that $\sigma=Id$, let $\mu$ be the invariant distribution of the associated Ornstein-Uhlenbeck process, i.e. $\mu:= N(0, \frac{1}{2}(-A)^{-1})$, and set $\nu :=Z e^{-V}\cdot \mu$ , where $Z$ is the normalizing constant $(\int_H e^{-V} \;d\mu)^{-1}$ so that $\nu$ is a probability.

By \cite{DaFlRoVe16}, Section 7, it follows that the so chosen quadruplet $(A ,F, \sigma, \nu)$ satisfies {\bf Hypothesis 1} and {\bf Hypothesis 2}.

Also, note that because $(-A)^{-1}$ is non-degenerate, $supp(\mu)=supp(\nu)=H$, i.e. using the notation from Subsection 1.1, we have that $H_0=H$.

The main result of this paragraph is the following.

\begin{coro}
\begin{enumerate}[(i)]
\item \thnameref{thm 3.4} and \thnameref{thm 3.16} apply for $(A ,F, \sigma, \nu)$ and $H_0=L^{2}(0,1)$.
\item Let $M$ be either $L^{2m}(0,1)$ or $C([0,1])$. 
If $x\in M$ then there exists a unique probabilistically weak solution (in the mild sense) to equation \eqref{eq 3.1}  which starts from $x$ and remains in $M$.
Consequently, if $Y$ is the process given by \thnameref{thm 3.16}, 
then the set $H \setminus M$ is polar, i.e for any $x\in L^{2}(0,1)$
\begin{equation*}
\mathbb{Q}^{x}(\{ Y(t) \in M \mbox{ for all } t>0\})=1.
\end{equation*}
\end{enumerate}
\end{coro}
\begin{proof}
Since the first assertion is clear by the previous discussion, let us prove the second one.
By \cite{Da04}, Theorem 4.8, we have that if $x\in M$ then equation \eqref{eq 3.1} with $B=0$ has a unique strong solution (in the mild sense) which starts from $x$ and which remains in $M$. 
Clearly, the solution is Markov and it's semigroup is precisely $(P_t)_{t\geq 0}$ provided by \thnameref{thm 3.2}.

Further, by Girsanov transformation we have that if $x\in M$, then equation \eqref{eq 3.1} (general $B$) has a (unique) pobabilistically weak solution (in the mild sense) which starts from $x$ and remains in $M$.

Finally, we can apply Theorem \ref{thm 2.2} and Corollary \ref{coro 1.12} (for the resolvent $\mathcal{V}$ in \thnameref{thm 3.16}) to conclude that $H\setminus M$ is polar for $Y$.

\end{proof}

\subsection{A natural extension of Markov processes}
Throughout this subsection we place ourselves in the following general situation: $(E,\mathcal{B})$ is a Lusin measurable space (i.e., it is measurable isomorphic to a Borel subset of a metrizable compact space endowed with the Borel $\sigma$-algebra), and $M \in \mathcal{B}$ is a subset of $E$.
Further, we assume that there exists a right Markov process $X=(\Omega, \mathcal{F}, \mathcal{F}_t, X(t), \theta(t), \mathbb{P}^x)$ with state space $M$ and resolvent family $\mathcal{U}=(U_\alpha)_{\alpha>0}$.
In particular, the process $X$ starts and remains in $M$.
We remark that no topology is a priori given on $E$.
This is because the resolvent $\mathcal{U}$ comes with its own topology, the so called fine topology, and the continuity properties of the paths of $X$ are regarded w.r.t. this topology; the reason is that if a right process has right continuous paths with respect to some given Lusin topology $\tau$ (see Definition \ref{defi 3.3}) whose Borel $\sigma$-algebra is $\mathcal{B}$, then $\tau$ is automatically coarser than the fine topology (see Appendix for details).
For simplicity and in spite of the applications considered in the previous subsection, we assume that the lifetime of the process is infinite; nevertheless, the results of this subsection remain true when $X$ has finite lifetime.
Our aim here is to extend $X$ to a right Markov process $\overline{X}$ on the entire space $E$ in such a way that when $\overline{X}$ starts from $M$ it evolves like $X$, and when it starts from $E\setminus M$ it will immediately enter $M$, from where it continues to evolve like $X$. 
Formally, we have the following definition, which is central for most of the work done in this paper.

\begin{defi} \label{defi 2.1}
We say that a Markov process $\overline{X}=(\overline{\Omega}, \overline{\mathcal{F}}, \overline{\mathcal{F}}_t, \overline{X}(t), \overline{\theta}(t), \overline{\mathbb{P}}^x)$, with state space $E$, 
is a {\it natural extension} of $X$ if the following conditions are fulfilled.
\begin{enumerate}[(i)]
\item $\overline{X}$ is a right process.
\item The processes $(({X}(t))_{t \geq 0}, \mathbb{P}^{x})$ and $((\overline{X}(t))_{t\geq 0}, \overline{\mathbb{P}^{x}})$ are equal in distribution for all $x\in M$;
\item For every $x\in E$ one has $\overline{\mathbb{P}^{x}}$-a.s. $\overline{X}(t) \in M$ for all $t>0$, i.e. $E\setminus M$ is polar w.r.t. $\overline{\mathcal{U}}$.
\end{enumerate}
\end{defi}

As it will be seen in the proofs of the main results from the previous subsection, a direct pathwise extension is not always possible; instead, it is much more at hand to extend the one dimensional distributions of the process. In fact, from a potential theoretical point of view (see Appendix), what we need is an extension of the resolvent, in the following sense.

\begin{defi} \label{defi 2.2}
A sub-Markovian resolvent of kernels $\overline{\mathcal{U}}:=(\overline{U}_\alpha)_{\alpha>0}$ on $E$ is called an extension of $\mathcal{U}$ if:
 \begin{enumerate}[(i)]
 \item $\overline{U}_\alpha (1_{E \setminus M})=0$.
 \item $(\overline{U}_\alpha f)|_M=U_\alpha (f|_M) $ (on $M$) for all $\alpha>0$ and $f \in b\mathcal{B}$.
 \end{enumerate}
\end{defi}

It is easy to see that Definition \ref{defi 2.2} is consistent with Definition \ref{defi 2.1}. More precisely, we have:

\begin{prop} \label{prop 2.3}
If $\overline{X}$ is a natural extension of $X$, then its resolvent denoted by $\overline{\mathcal{U}}$ is an extension of $\mathcal{U}$.
\end{prop}

The aim of this subsection is to investigate the converse of Proposition \ref{prop 2.3}, namely: if $\overline{\mathcal{U}}$ is an extension of $\mathcal{U}$, under which conditions $\overline{\mathcal{U}}$ is the resolvent of a natural extension $\overline{X}$ of $X$?
To answer this question, we need to consider the following condition which is a version of the assumption (H1) -- (H3) from \cite{BeRo11a}, page 846.

\vspace{0.2 cm}
\noindent
{\bf (H)}  There exists a min-stable convex cone $\mathcal{C}\subset bp\mathcal{B}$ such that
\begin{enumerate}[(i)]
\item $1 \in \mathcal{C}$ and $\sigma(\mathcal{C})=\mathcal{B}$.
\item For some (hence all) $\beta >0$ we have $\overline{U}_\beta f \in \mathcal{C}$  for all $f \in \mathcal{C}$.
\item $\lim\limits_{\alpha \to \infty}\alpha \overline{U}_\alpha f=f$ point-wise on $E$ for all $f \in \mathcal{C}$.
\end{enumerate}

We are now in the position to present the main result of this subsection.
\begin{thm} \label{thm 2.2}
Let $\overline{\mathcal{U}}$ be an extension of $\mathcal{U}$.  Then there exists a natural extension $\overline{X}$ of $X$, with resolvent $\overline{\mathcal{U}}$, if and only if (H) is satisfied.
\end{thm}

A consequence of (the proof of) the previous theorem is that any natural extension of $X$, if exists, is unique in distribution:
\begin{coro} \label{coro 1.12}
Any extension $\overline{\mathcal{U}}$ of $\mathcal{U}$ which satisfies (H), is uniquely determined. 
In particular, any natural extension of $X$ is unique in distribution.
\end{coro}

\vspace{0.2 cm}
{\noindent \bf Further results concerning natural topologies.}
Until the next paragraph on general remarks about condition (H), we assume that the latter is fulfilled, so that $\overline{\mathcal{U}}$ is the resolvent of a right Markov process $\overline{X}$ on $E$, which is a natural extension of $X$, as in Theorem \ref{thm 2.2}.
In addition, we suppose that we are given a Lusin topology $\tau$ on $E$, whose Borel $\sigma$-algebra coincides with $\mathcal{B}$, such that $X$ has a.s. right continuous paths w.r.t. $\tau$.
By Definition \ref{defi 4.1}, Theorem \ref{thm 4.6}, and Corollary \ref{coro 4.10} from Appendix, it means that $\tau$ is a natural topology on $M$ w.r.t. $\mathcal{U}$.
 
Unfortunately, we can not say that $\tau$ remains a natural topology on $E$ w.r.t. $\overline{\mathcal{U}}$, without further assumptions.
Nevertheless, at least on $M$, we can show that $\overline{X}$ inherits the same path-continuity properties as $X$.
In fact, because $E\setminus M$ is polar, we can say a little bit more:

\begin{prop}  \label{prop 2.9}
The following assertions hold.

\begin{enumerate}[(i)]
\item The paths $(0,\infty)\ni t\mapsto\overline{X}(t)$ are $\overline{\mathbb{P}}^{x}$-a.s. right continuous w.r.t. $\tau$, for all $x\in E$; in addition, if $x\in M$, then the paths are continuous in $0$ w.r.t. $\tau$.
\item If $X$ has paths with left limits in $M$ (or $E$) w.r.t. $\tau$, then so does $\overline{X}$.
\item If $X$ has continuous trajectories on $M$, then the paths $ (0,\infty)\ni t\mapsto \overline{X}(t)$ are $\overline{\mathbb{P}}^{x}$-a.s. continuous w.r.t. $\tau$, for all $x\in E$.
\end{enumerate}
\end{prop}
 
In view of Proposition \ref{prop 2.9},  if $x\in E\setminus M$, the $\overline{\mathbb{P}}^{x}$-a.s. $\tau$-continuity in $0$ of the paths of $\overline{X}$ is a more delicate issue, which requires more information about how the fine topology is related to the given topology $\tau$.  Our next aim is to discuss some general conditions which allow us to tackle this issue. However, because of their generality, in certain concrete applications like those studied in Subsection 1.1 (where $\tau$ is the $|\cdot|$-topology), one has to use specific tools in order to show that $\tau$ is a natural topology.
Throughout, $C_b(E)$ denotes the space of real valued, bounded and $\tau$-continuous functions on $E$.

First of all, under a minimal extra condition, we can say a bit more on how $\tau$ can be related, in general, to the fine topology associated to $\overline{\mathcal{U}}$.

\begin{prop} \label{prop 2.10}
In addition to (H), assume that there exists a vector lattice $\mathcal{C}' \subset b\mathcal{B}$ possessing a countable subset which separates the points of $E$, such that $\overline{U}_\alpha \mathcal{C}' \subset C_b(E)$ for all $\alpha >0$. Then one can choose a natural (in fact, Ray) topology $\tau_0$ on $E$ (w.r.t. $\overline{\mathcal{U}})$ which is smaller than the given topology $\tau $.
\end{prop} 
 
Concerning the case when $\tau$ itself is a natural topology, we make first the following observation.

\begin{rem} \label{rem 2.6}
Assume that $\mathop{\lim}\limits_{\alpha \to \infty} \|\alpha \overline{U}_\alpha f-f\|_\infty=0$ for all $f \in \widetilde{\mathcal{C}}$, where $\widetilde{\mathcal{C}} \subset C_b(E)$ is such that it generates the topology $\tau$ on $E$.
Then $\tau$ is a natural topology on $E$. 
\end{rem}

Unfortunately, the uniform convergence assumption from Remark \ref{rem 2.6} is difficult to check in many situations, even in finite dimensions.
We therefore turn our attention to a more practical situation:

\vspace{0.2 cm}
{\noindent \bf Assumption.}  There exists a positive measure $\nu$ on $E$, with full support, and $\overline{\mathcal{U}}$ regarded as a family of operators on $L^{p}(\nu)$ for some $1\leq p < \infty$, is the resolvent of an $m$-dissipative operator $({\sf L}, D({\sf L}))$ on $L^{p}(\nu)$. 
In addition, suppose that $\overline{\mathcal{U}}$ is $L^{p}$-strong Feller, i.e. $\overline{U}_\alpha (L^{p}(\nu)) \subset C(E)$ for one (hence all) $\alpha > 0$.

\begin{rem}
Under the above assumption, since $D(L)=\overline{U}_\alpha (L^{p}(E))$, $\alpha>0$, it is clear that any element $f\in D(L)$ has a continuous version on $E$.
\end{rem}

\begin{prop} \label{prop 2.7}
Consider that the above assumption is fulfilled. Then each function $f\in C(E)\cap D(L)$ is finely continuous. 
In particular, if $\mathcal{A}$ is a countable subset of functions from $C(E)\cap D({\sf L})$ which separates the points of $E$, then the (initial) topology generated by $\mathcal{A}$ is a natural topology.
\end{prop} 
 
\noindent
{\bf General remarks on condition (H).}

\begin{rem} \label{rem 2.3}
Let $\overline{\mathcal{U}}$ be an extension of $\mathcal{U}$.
\begin{enumerate}[(i)]
\item By Lusin theorem, if $\mathcal{C} \subset b\mathcal{B}$ contains a countable subset which separates the points of $E$, then $\sigma(C)=\mathcal{B}$; in this case, condition (H), (i) reduces to $1\in \mathcal{C}$.
\item If $\mathcal{C} \subset C_b(E)$ is a min-stable convex cone and $\overline{\mathcal{U}}$ is Feller, then condition (H), (ii) becomes:
for some $\beta >0$ one has $U_\beta(\mathcal{C}|_M)\subset \mathcal{C}|_M$.
\item If (H) holds, $\overline{U}_\alpha(\mathcal{C})\subset
C_b(E)$ for all $\alpha>0$, and $\alpha U_\alpha(f|_M)\rightarrow f$ uniformly on $M$ when $\alpha \to \infty$ for all $f\in \mathcal{C}$, then $\mathcal{C} \subset C_b(E)$.
Indeed, observe first that if a sequence $(x_n)_n \subset M$ is converging to $x\in E$, then $\mathop{\lim}\limits_n f(x_n)=f(x)$ if $f\in \mathcal{C}$:
$$
|f(x)-f(x_n)|\leq |f(x)-f_k(x)|+|f_k(x)-f_k(x_n)|+ |f_k(x_n)-f(x_n)|,
$$
where $f_k:=k\overline{U}_k(f|_M)$, and by the uniformly convergence assumption, there exists $k_0\in\mathbb{N}$ such that if $k\geq k_0$ then $|f_k(x_n)-f(x_n)|<\varepsilon$ for all $n$.
To conclude that $f$ is a continuous function on $E=\overline{M}$, we can argue now as in the proof of Remark 1.1 from \cite{Be11}.
\item Let $\mathcal{A}$ be the closure in the supremum norm of the linear space spanned by $b\mathcal{E}(\overline{\mathcal{U}}_\beta)$. 
If $\mathcal{C} \subset \mathcal{A}$, then condition (H), (iii) holds.
Note that $\mathcal{C} \subset \mathcal{A}$ provided that 
\newline $\mathop{\lim}\limits_{\alpha \to \infty}\|\alpha \overline{U_\alpha (f|_M)}-f\|_{\infty} =0$ for all $f\in \mathcal{C}$.
\end{enumerate}
\end{rem}

\begin{lem} \label{lem 2.4}
Assume that $\overline{\mathcal{U}}$ is an extension of $\mathcal{U}$, and that (H),(ii) is satisfied. If $\mathcal{C} \subset C_b(E)$ and there exist  $\beta \geq 0$ such that for each $f\in \mathcal{C}$ the family $(\alpha \overline{U}_{\alpha+\beta}f)_{\alpha>0}$ is equicontinuous on $E$, then (H),(iii) holds.
\end{lem}
 
A typical situation when (H) is automatically fulfilled is as follows; see Subsection 1.1, \thnameref{thm 3.2}.

\begin{coro} \label{prop 2.5}
Assume that $(E,d)$ is a Polish metric space and let $Lip_b(E)$ denote the space of all bounded Lipschitz functions on $E$. 
Suppose that $\overline{\mathcal{U}}_\alpha (Lip_b(E))\subset Lip_b(E)$ for some (hence all) $\alpha >0$, and that there exist  $\beta \geq 0$ such that for each $f\in Lip_b(E)$, the family $(\alpha \overline{U}_{\alpha+\beta}f)_{\alpha>0}$ is equicontinuous on $E$.

Then (H) holds with $\mathcal{C}:=\{f\in Lip_b(E): f\geq 0\}$. 
\end{coro}

\subsection{A typical application of Theorem \ref{thm 2.2}}
By \cite{MaRo92} (see \cite{FuOsTa11} for the symmetric case, but also \cite{St99b} for a generalized theory), if $(\mathcal{E}, D(\mathcal{E}))$ is a (quasi-regular) Dirichlet form on $L^{2}(E, \nu)$, then one can always associate a standard (in particular, right) process $X$, whose transition function, regarded on $L^{2}(E, \nu)$, coincides with the $C_0$-semigroup generated by $\mathcal{E}$. 
A specific issue of the powerful Dirichlet forms approach of constructing Markov processes is that, in general, there is a $\nu$-exceptional set which has to be removed from the space $E$ in order to construct a right process which solves the martingale problem. In order to obtain a process on the entire space, a typical artificial extension is performed: if the process starts from the exceptional set, it is forced to remain stuck. 
This is usually called the {\it trivial} extension, and it is clearly in contrast with our natural extension considered in the previous subsection, where if the process starts from a "bad" point, it will immediately return to the set of "good" points, from where it follows the dynamic governed by the infinitesimal generator. 
So a natural question arises: given a preferred sub-Markovian resolvent of kernels $\mathcal{U}$ on $E$, which is associated (up to $\nu$-classes) to a quasi-regular Dirichlet form $\mathcal{E}$, can we construct a right process with resolvent $\mathcal{U}$, without further modifications?

Motivated by (and forgetting of) the aforementioned context of Dirichlet forms, we turn now to the following general situation: let $\mathcal{U}$ be a sub-Markovian resolvent of kernels on a Lusin topological space $(E, \tau)$, fulfilling the following assumptions:
\begin{enumerate}[(i)]
\item Condition (H) holds with $\mathcal{U}$ instead of $\overline{ \mathcal{U}}$.
\item There exists a reference measure $\nu$ on $\mathcal{B}$ for $\mathcal{U}$, i.e. if $\nu (A)=0$ then $U(1_A)\equiv 0$ for all $A\in \mathcal{B}$.
\item There exists a countable measure separating subset $\mathcal{A}_0 \subset b\mathcal{B}$ such that  ${U}_\alpha (\mathcal{A}_0) \subset C_b(E)$ for all $\alpha >0$.
\end{enumerate}

\begin{defi}
A  sub-Markovian resolvent of kernels $\widetilde{\mathcal{U}}:=(\widetilde{U}_\alpha)_{\alpha>0}$ on $E$ is called a {\it $\nu$-version} of $\mathcal{U}$ if ${U}_\alpha f=\widetilde{U}_\alpha f$ $\nu$-a.e. for all $f\in b\mathcal{B}$, $\alpha >0$.
\end{defi}

Our next result shows that once we know that a $\nu$-version of $\mathcal{U}$ (e.g. one obtained by a trivial extension) has associated a right process, then so does $\mathcal{U}$.

\begin{coro} \label{coro 3.1}
Assume that $\mathcal{U}$ has a $\nu$-version $\widetilde{\mathcal{U}}$ which has associated a right Markov process $\widetilde{X}$ on $E$ with a.s. $\tau$-right continuous paths.
Then there exists a right Markov process $X$ on $E$, with resolvent $\mathcal{U}$.
In fact, $X$ is the natural extension of the restriction of $\widetilde{X}$ from $E$ to a smaller set $M\subset E$.
In particular, Propositions \ref{prop 2.9} and \ref{prop 2.10} apply.
\end{coro} 
 
\section{Proofs of the results from Section 1}

\subsection{Proofs of the results from Subsection 1.2}
\begin{proof}[Proof of Theorem \ref{thm 2.2}]
For the direct implication, we use only that $\overline{X}$ is a right process with resolvent $\overline{\mathcal{U}}$, because by Definition \ref{defi 4.5} and Remark \ref{rem 4.6}, (ii) from Appendix, we can take $\mathcal{C}$  to be a Ray cone w.r.t. $\overline{\mathcal{U}}$.

Assume now that condition (H) is satisfied. One can see that the restriction to $M$ of any $\overline{\mathcal{U}}_\beta$-excessive function is $\mathcal{U}_\beta$-excessive and the converse also holds:

\vspace{0.2 cm}
\noindent
(1)  \hspace{0.2 cm} Any $\mathcal{U}_\beta$-excessive funtion $w$ has a unique extention $\overline{w}$ to $E$ which is $\overline{\mathcal{U}}_\beta$-excessive, hence
 $\mathcal{E}(\mathcal{U}_\beta)=\mathcal{E}(\overline{\mathcal{U}}_\beta)|_M$.

\vspace{0.2 cm}
\noindent
Indeed, let $w\in \mathcal{E}(\mathcal{U}_\beta)$. We may assume that $w \leq 1$. 
Then the function $w_1$ extending $w$ with the value $1$ on $E \setminus M$ is $\overline{\mathcal{U}}_\beta$-supermedian and the $\overline{\mathcal{U}}_\beta$
-excessive regularization of $w_1$ is the $\overline{\mathcal{U}}_\beta$-excessive function extending $w$ from $M$ to $E$; see Appendix, right after Definition \ref{defi 4.1}.

We prove now that $\overline{\mathcal{U}}$ satisfies the assumption (H') from the beginning of Appendix, i.e.:

\vspace{0.2 cm}
\noindent
(2) \hspace{0.2 cm} $\sigma(\mathcal{E}(\overline{\mathcal{U}}_\beta))=\mathcal{B}$ and all the points of $E$ are non-branch points with respect to $\overline{\mathcal{U}}_\beta$,

\noindent
that is $1 \in \mathcal{E}(\overline{\mathcal{U}}_\beta)$ and if $u,v \in \mathcal{E}(\overline{\mathcal{U}}_\beta)$ then $\inf(u,v)=\widehat{\inf(u,v)}$.

The proof of (2) follows essentially the proof of Proposition 2.1 from \cite{BeRo11a}.
Indeed, first of all note that $\sigma (\overline{U}_\beta(b\mathcal{B}))\subset \sigma (\mathcal{E}(\overline{\mathcal{U}}_\beta)) \subset \mathcal{B}$.
On the other hand, by (H)-(iii) and the resolvent equation we get that any element from $\mathcal{C}$ is a point-wise limit of functions from $\overline{U}_\beta(b\mathcal{B})$, hence ($\mathcal{B}=$) $\sigma(\mathcal{C})\subset \sigma (\overline{U}_\beta(b\mathcal{B}))$, so the first assertion follows.
Next, if $f,g \in \mathcal{C}$ then the function $w:=inf(\overline{U}_\beta f,\overline{U}_\beta g)$ is $\overline{\mathcal{U}}_\beta$-supermedian and belongs to $\mathcal{C}$ by (H). 
On the other hand, by (H)-(iii), we get that $w=\lim\limits_{\alpha \to \infty} \alpha \overline{U}_\alpha w$, hence $w$ is $\overline{\mathcal{U}}_\beta$-excessive. 
Now, using (H)-(i) and Lemma 1.2.10 from \cite{BeBo04a}, we get that the set of all non-branch points w.r.t. $\overline{\mathcal{U}}_\beta$ is $E$.

Let $(M_1,\mathcal{B}_1)$ be the saturation of $M$, and $\mathcal{U}_1$ be the extension of $\mathcal{U}$ to $M_1$, given by relation (\ref{eq 4.1}) from Appendix. 
The next step is to show that:

\vspace{0.2 cm}
\noindent
(3) \hspace{0.2 cm} The map $E \ni x \overset{j}{\mapsto}\delta_x \circ \overline{U}_\beta \in Exc(\mathcal{U}_\beta)$ is a measurable embedding of $E$ into $M_1$, $j(\mathcal{B})=\mathcal{B}_1|_{j(E)}$, and $\overline{\mathcal{U}}$ is the restriction of $\mathcal{U}^{1}$ from $M_1$ to $E$, $\mathcal{E}(\overline{U}_\beta)=\mathcal{E}(\mathcal{U}^{1}_\beta)|_E$.

\noindent
Indeed, we already observed that $\delta_x \circ \overline{U}_\beta$ is a measure on $M$ and one can see that it belongs to $Exc(\mathcal{U}_\beta)$. 
Recall that by \cite{St89}, (2) implies that the {\it specific solidity of potentials} holds in $Exc(\overline{\mathcal{U}}_\beta)$: if $\xi_1, \xi_2, \mu \circ \overline{U}_\beta \in Exc(\overline{\mathcal{U}}_\beta),\xi_1+\xi_2=\mu\circ\overline{U}_\beta$, then there exist two measures $\mu_1, \mu_2$ on $E$ such that $\xi_i=\mu_i\circ\overline{U}_\beta$, $i=1,2$.
Consequently, $\delta_x \circ \overline{U}_\beta \in M_1$.
The injectivity of $j$ follows from (H).
To see that $j$ is $\mathcal{B}/\mathcal{B}_1$-measurable, it is sufficient to prove that for every $v \in \mathcal{E}(\mathcal{U}_\beta)$ the function $\widetilde{v}\circ j$ is $\mathcal{B}$-measurable.
If $x \in E$ then $\widetilde{v}\circ j (x)= L^{\beta}(\delta_x\circ \overline{U}_\beta, v)= \mathop{\sup}\limits_\alpha L^{\beta}(\delta_x\circ \overline{U}_\beta, \alpha U_{\beta + \alpha}v)= \mathop{\sup}\limits_\alpha L^{\beta}(\delta_x\circ \alpha \overline{U}_{\beta+\alpha} \circ U_\beta, v) = \mathop{\sup}\limits_\alpha \alpha \overline{U}_{\beta+\alpha} v(x)=\overline{v}(x)$, hence $\widetilde{v}\circ j=\overline{v}$ and by $(1)$ we get $\widetilde{v}\circ j \in p\mathcal{B}$.
From Lusin's Theorem we conclude  that $j(\mathcal{B})=\mathcal{B}_1|_{j(E)}$.
If $\xi=\delta_x\circ \overline{U}_\beta$, $x\in E$, $f \in bp\mathcal{B}_1$, and $\alpha > 0$, then $\overline{U}_\alpha (f|_E)(x)=L^{\beta} (\delta_x \circ\overline{U}_\beta, U_\alpha (f|_M))=U^{1}_\alpha f(\xi)$, so, $\overline{\mathcal{U}}=\mathcal{U}^{1}|_E$.

From now on we identify $E$ with $j(E)$ and $\mathcal{B}$ with $j(\mathcal{B})$, so, $E \in \mathcal{B}_1$ and $\mathcal{B}=\mathcal{B}_1|_E$.

\vspace{0.2 cm}
\noindent
(4) \hspace{0.2 cm} There exists a Ray cone $\mathcal{R}$ associated with $\mathcal{U}_\beta$ s.t. $\overline{\mathcal{R}}:=\{\overline{v}: \; v \in \mathcal{R}\}$ separates the points of $E$. 
If $\mathcal{R}$ is a Ray cone associated with $\mathcal{U}_\beta$ and $\overline{\mathcal{R}}$ separates the points of $E$ then $\overline{\mathcal{R}}$ is a Ray cone associated with $\overline{U}_\beta$. 

\vspace{0.2 cm}
\noindent 
Indeed, let $\mathcal{F}_0 \subset b\mathcal{B}$ be countable and measure separating on $E$, which exists since $E$ is a Lusin space. Then taking into account the standard construction of a Ray cone (see Appendix, Remark \ref{rem 4.6}, (ii)), one can suppose that it includes the countable set $U_\beta(\mathcal{F}_0)$, and therefore $\overline{U}_\beta (\mathcal{F}_0)$ separates the points of $E$.
Let now $\mathcal{R}$ be a Ray cone with $\overline{R}$ separating the points of $E$.
We know by (2) that $\mathcal{E}(\overline{U}_\beta)$ is min-stable and by (1) we get that $\overline{U}_\beta((\overline{\mathcal{R}}-\overline{\mathcal{R}})_+)= \overline{U}_\beta(({\mathcal{R}}-{\mathcal{R}})_+) \subset \overline{\mathcal{R}}$ and that $\overline{\mathcal{R}}$ is min-stable.
Since $\overline{\mathcal{R}}$ is separable, separates the points of $E$ and $\sigma(\overline{R})\subset \mathcal{B}$ we deduce by Lusin's Theorem that $\sigma(\overline{\mathcal{R}})=\mathcal{B}$.

Note that by Corollary \ref{coro 4.10} from Appendix, $X$ remains a right process if we endow $M$ with any Ray topology.

We return now to the saturation $M_1$ of $M$. 
Let $\mathcal{R}$ be a Ray cone associated with $\mathcal{U}_\beta$ such that $\overline{\mathcal{R}}$ separates the points of $E$, which exists by (4).

One can equip $M_1$ with the Ray topology $\tau_1$ generated by $\widetilde{R}:= \{\widetilde{v}: \; v\in \mathcal{R}\}$. 
Then $M_1$ becomes a Lusin topological space and Theorem \ref{thm 4.15} from Appendix, there exists a right Markov process $X^{1}$ with state space $M_1$ having $\mathcal{U}^{1}$ as associated resolvent.
In addition, the set $M_1 \setminus M$ is a polar subset of $M_1$. 
Therefore the set $M_1 \setminus E$ is also a polar set and thus, $\overline{\mathcal{U}}$ is the resolvent of the restriction $\overline{X}$ of $X^{1}$ to the absorbing set $E$.

In conclusion:

\vspace{0.2 cm}
\noindent
(5) \hspace{0.2 cm} $\overline{X}$ is a right Markov process with state space $E$ endowed with the Ray topology $\tau_0:=\tau(\overline{\mathcal{R}})$ generated by $\overline{\mathcal{R}}$, and by Corollary \ref{coro 4.10} from Appendix, it remains a right process w.r.t. any natural topology.
Because $\mathcal{U}$ is the restriction to $M$ of $\overline{\mathcal{U}}$, and since the set $E\setminus M$ is polar with respect to $\overline{\mathcal{U}}$ (see Theorem \ref{thm 4.13} from Appendix), we get that $\overline{X}$ is a natural extension of $X$.

\end{proof}

\begin{proof}[Proof pf Corollary \ref{coro 1.12}]
By (3) from the proof of Theorem \ref{thm 2.2}, if $\overline{\mathcal{U}}$ and $\overline{\mathcal{U}}'$ are two extensions of $\mathcal{U}$ satisfying (H), then on the saturation $M_1$ of $M$ they are both the restriction of $\mathcal{U}^{1}$ to $E$, so they must coincide.

\end{proof}

\begin{proof}[Proof of Proposition \ref{prop 2.9}]
(i) First of all, note that $(\overline{X}(t), \overline{\mathbb{P}^{x}})$ is a right process on $M$ with resolvent $\mathcal{U}$.
By Appendix, if $x\in M$, it follows that $\overline{X}$ has $\mathbb{P}^{x}$-a.s. right continuous paths w.r.t. any natural topology (w.r.t. $\mathcal{U}$) on $M$, in particular w.r.t. $\tau$.

Let now $x\in E\setminus M$ and $s>0$. Then, by the Markov property
\begin{align*}
\overline{\mathbb{P}}^{x}\{t\mapsto\overline{X}(t) & \mbox{ is right continuous on } (s, \infty) \mbox{ w.r.t. } \tau\}=\\
&=\overline{\mathbb{P}}^{x}\{t\mapsto\overline{X}(t+s) \mbox{ is right continuous on } (0, \infty) \mbox{ w.r.t. } \tau\}\\
&=\overline{\mathbb{E}}^{x}\{1_{\{t\mapsto\overline{X}(t) \mbox{ is right continuous on } (0, \infty) \mbox{ w.r.t. } \tau\}}\circ \theta(s)\}\\
&=\overline{\mathbb{E}}^{x}\{\mathbb{P}^{\overline{X}(s)}\{t\mapsto\overline{X}(t) \mbox{ is right continuous on } (0, \infty) \mbox{ w.r.t. } \tau\}\}\\
&=1,
\end{align*}
where the last equality holds because $\overline{\mathbb{P}}^{x}\{\overline{X}(s)\in M\}=1$, by Definition \ref{defi 2.1}.
Since $s>0$ was arbitrarily chosen, the proof of (i) is complete.

(ii)-(iii) Let $D$ be a countable dense subset of $\mathbb{R}_+$.
Endowed with the $\sigma$-algebra generated by the canonical projections, the power sets $E^{D}$ and $M^{D}$ are a Lusin measurable spaces; in fact, $M^{D}$ is a measurable subset of $E^{D}$. 
Let $W^{M}$ (resp. $W^{E}$)  denote the restrictions to $D$ of the c\`adl\`ag paths from $\mathbb{R}_+$ to $M$ (resp. $E$).
Now, because $\tau$ is a Lusin topology, by \cite{DeMe78}, Chapter IV, pages 91-92, we have that $W^{M}$ and $W^{E}$ are measurable subsets of $M^{D}$ and $E^{D}$, respectively.
The fact that the set $W^{0}$ of the restrictions to $D$ of all $\tau$-continuous paths in $M$ is measurable, is well known and straightforward.

Suppose now that $X$ has paths with left limits in $M$.  
Let us first take $x\in M$, and let $\overline{\mathbb{P}}^{x}\circ \overline{X}^{-1}$ and $\mathbb{P}^{x}\circ X^{-1}$ be the laws on $M^{D}$ of the processes $\overline{X}$ and $X$, respectively. 
By Definition \ref{defi 2.1}, we have $\overline{\mathbb{P}}^{x}\circ \overline{X}^{-1}=\mathbb{P}^{x}\circ X^{-1}$. 
It follows that $\overline{\mathbb{P}}^{x}\circ \overline{X}^{-1}$ is supported on $W^{M}$, and
since we know from (i) that the paths of $\overline{X}$ are $\overline{\mathbb{P}}^{x}$-a.s. right continuous w.r.t. $\tau$, we conclude that they have also $\overline{\mathbb{P}}^{x}$-a.s. left limits in $M$ w.r.t. $\tau$. 

The other two cases follow similarly, by replacing $W^{M}$ with $W^{E}$ and $W^{0}$, respectively.
\end{proof}

\begin{proof}[Proof of Proposition \ref{prop 2.10}]
If we start with a countable subset $\mathcal{A}_0 \subset \mathcal{C}'$ which separates the points of $E$, then as in Remark \ref{rem 4.6}, (ii) from Appendix, one can easily construct a Ray cone $\overline{\mathcal{R}}$ (w.r.t. $\overline{\mathcal{U}}$) such that $\overline{\mathcal{R}}\subset C_b(E)$, and we can take $\tau_0$ to be it's Ray topology.

\end{proof}

\begin{proof}[Proof of Proposition \ref{prop 2.7}]
If $f \in C(E)\cap D(L)$, there exists a measurable function $g$ on $E$ which is from $L^{p}(\nu)$ such that $f=\overline{U}_1g$ $\nu$-a.e. 
Set $v_1:=\overline{U}_1g^{+}$ and $v_2:=\overline{U}_1g^{-}$.
Then $v_1$ and $v_2$ are $1$-excessive functions which are continuous (hence finite) on $E$. Since $\nu$ has full support and $f\in C(E)$, it follows that $f=v_1-v_2$ on $E$, hence $f$ is finely continuous.

Next, if $\mathcal{A}\subset C(E)\cap D(L)$ is countable and separates the points of $E$, it follows that the topology generated by $\mathcal{A}$ is a Lusin topology. The second part of the statement follows now by the first one.

\end{proof}

\begin{proof}[Proof of Lemma \ref{lem 2.4}]
Since $\mathcal{U}$ is the resolvent of a right process on $M$, we have that $\alpha U_{\alpha + \beta}(f|_M) \underset{\alpha \to \infty}{\longrightarrow} f|_M$, hence $\alpha \overline{U}_{\alpha + \beta}f|_M \underset{\alpha \to \infty}{\longrightarrow} f|_M$ point-wise on $M$ for all $f\in \mathcal{C}$.

Let now $x \in E\setminus M$. 
Since $f\in \mathcal{C}$ is continuous and the family $(\alpha \overline{U}_{\alpha+\beta} f)_{\alpha > 0} $ is equicontinuous on $E$, for arbitrarily fixed $\varepsilon > 0$ we have that there exist a neighbourhood $V(x)$ s.t. $|\alpha \overline{U}_{\alpha+\beta}f(x)-\alpha \overline{U}_{\alpha+\beta}f(y)|+|f(x)-f(y)|\leq \varepsilon$ for all $\alpha > 0$ and $y \in V(x)$.
Therefore,
\begin{align*}
\alpha \overline{U}_{\alpha+\beta}f(x)-f(x)| &\leq |\alpha \overline{U}_{\alpha+\beta}f(x)-\alpha \overline{U}_{\alpha+\beta}f(y)| +|\alpha \overline{U}_{\alpha+\beta}f(y)-f(y)|+ |f(x)-f(y)|\\
&\leq \varepsilon + |\alpha \overline{U}_{\alpha+\beta}f(y)-f(y)|, \;\; y \in V(x), \alpha>0.
\end{align*}
Because $M$ is dense in $E$, there exists $y \in V(x) \subset M$, so by taking limits in the above expression we get that $\mathop{\limsup}\limits_{\alpha\to \infty}|\alpha \overline{U}_{\alpha+\beta}f(x)-f(x)|\leq \varepsilon$.
But $\varepsilon$ was arbitrarily chosen, so $\alpha \overline{U}_{\alpha+\beta}f(x)\underset{\alpha \to \infty}{\longrightarrow}f(x)$. 
By the resolvent equation we obtain the same convergence for all $\beta >0$.

\end{proof}

\begin{proof}[Proof of Proposition \ref{prop 2.5}]
Clearly $\mathcal{C}$ is a min-stable convex cone, $\mathcal{C} \subset bp\mathcal{B}$, and satisfies (H),(i)-(ii). 
The third condition of (H) follows by Lemma \ref{lem 2.4}.

\end{proof}

\subsection{Proof of Corollary \ref{coro 3.1}}

Let $f \in \mathcal{A}_0$ and $\alpha >0$. 
Since ${U}_\alpha f$ is $\widetilde{\mathcal{U}}$-finely continuous (since it is $\tau$-continuous and $\tau$ is natural w.r.t. $\widetilde{\mathcal{U}}$) 
it follows that the set $V:=[{U}_\alpha f \neq \widetilde{U}_\alpha f]$ is $\widetilde{\mathcal{U}}$-finely open and $\nu$-negligible, hence $\nu$-polar w.r.t. $\widetilde{\mathcal{U}}$, i.e. $\nu(\widetilde{R}_\alpha^V 1)=0$; see Appendix for the definition of the reduced function.
If we set $M:=[\widetilde{R}_1^{V}1=0]$, then $M^c$ is $\nu$-inessential 
(i.e. $\widetilde{R}_1^{M^{c}}1=0$ on $M$ and $M^c$ is $\nu$-negligible) and $V\subset M^c$.
Since $\mathcal{A}_0$ is countable we can take a  $\nu$-inessential set $M^c$ as above such that ${U}_\alpha f = \widetilde{U}_\alpha f$ on $M$ for all $f\in \mathcal{A}_0$ and $\alpha \in \mathbb{Q}^{+}$.
Moreover, since $M^c$ is $\nu$-negligible and $\nu$ is a reference measure for ${\mathcal{U}}$, it follows that ${U}_\alpha 1_{E\setminus M} = 0$ on $E$. 

Let $\mathcal{U}^{'}$ denote the restriction of $\widetilde{\mathcal{U}}$ from $E$ to $M$.
Since $M^{c}$ is $\nu$-inessential, it is well known that $\mathcal{U}^{'}$ is the resolvent of the right Markov process on $M$ obtained by taking the restriction of $\widetilde{X}$ to $M$.
Moreover, for each $\alpha\in \mathbb{Q}^{+}$ we have that $({U}_\alpha f)|_M=U^{'}_\alpha (f|_M)$ for all $f\in \mathcal{A}_0$, and since $\mathcal{A}_0$ is measure separating, it follows that $U_\alpha|_M=U_\alpha^{'}$.
By the resolvent equation it follows that ${\mathcal{U}}|_M=\mathcal{U}^{'}$.

To conclude, we obtained that $\mathcal{U}$ is an extension of $\mathcal{U}^{'}$ from $M$ to $E$, for which condition (H) is fulfilled. 
Therefore, the statement follows by Theorem \ref{thm 2.2}.

\subsection{Proofs of the results from Subsection 1.1.} 
For simplicity, throughout this subsection we assume that $\sigma$ is the identity operator on $H$.
Also, note that by \thnameref{thm 3.2}, (i.1), and Remark \ref{rem 4.2} from Appendix, it follows that both conditions (H) and (H') from Section 2 and Appendix hold true for $\mathcal{U}$.

\begin{proof}[Proof of \thnameref{thm 3.4}]
First of all, as in the proof of Corollary 1.10 from \cite{DaRoWa09}, by \cite{BoDaRo96}, Theorem 1.1 and \cite{St99a}, Chapter II, Theorem 1.9 and Proposition 1.10, 
there exists a normal Markov process 
$X'=(\Omega', \mathcal{F}', \mathcal{F}'_t , X'(t) , \theta'(t),  \mathbb{P'}^x )$ with continuous trajectories in $H_0$, whose resolvent is (merely) a $\nu$-version of $\mathcal{U}$. 
Consequently, we are in the situation of Corrolary \ref{coro 3.1} so that there exists $\widetilde{M}\subset H_0$ such that $H_0\setminus \widetilde{M}$ is polar, 
and there exists a right Markov process 
$X=(\Omega, \mathcal{F}, \mathcal{F}_t , X(t), \theta(t), \mathbb{P}^x)$ on $H_0$ such that if the process starts from $\widetilde{M}$ then it has continuous trajectories in $\widetilde{M}$ w.r.t. the $|\cdot|$-topology.

What we are going to show later on in (ii) is that, in fact, the process $X$ has $|\cdot|$-continuous trajectories on $[0, \infty)$ for all starting points in $H_0$.
Before we proceed to the proofs of (i-iii), set 
$$
\widetilde{A}=A-\omega I \mbox{ and }\widetilde{F}_0=F_0+\omega I,
$$
and $(e_k)_{k\geq 1} \subset H$ an orthonormal basis consisting of eigenvectors of $\widetilde{A}$, with corresponding eigenvalues $(-\lambda_k)_{k\geq 1}\subset (-\infty, 0)$.
For $k\geq 1$ let $$\varphi_k(x)=\langle x, e_k \rangle, x\in H,$$
and note that since $\mathcal{F}\mathcal{C}_b^{2} \subset D(L)$ (see Remark \ref{rem 3.3}), by localization we get for all $k\geq 1$
\begin{align*}
\varphi_k, \varphi_k^{2} \in D(L) \mbox{ and } & L\varphi_k=-\lambda_k \varphi_k + \langle \widetilde{F}_0, e_k \rangle\\
& L\varphi_k^{2}=-2\lambda_k \varphi_k^{2} + 2\varphi_k \langle \widetilde{F}_0, e_k \rangle + 2.
\end{align*}
\vspace{0.2 cm}
\noindent
(i). Let $g(x):=(1+|x|^{2})(1+|F_0(x)|^{2}), x\in H_0.$
Then $g\in L^1(\nu)$ by assumption, hence $U_\alpha g \in L^1(\nu)$ and $U_\alpha g < \infty$ $\nu$-a.e.
Since $\nu$ is a reference measure, it is well known that $[U_\alpha g=\infty]$ is a polar set for all $\alpha >0$, hence the set
$M_0:=E\setminus (\mathop{\bigcup}\limits_{\alpha>0}[U_\alpha g=\infty])=E\setminus (\mathop{\bigcup}\limits_{\alpha \in \mathbb{Q}_{+}^{\ast}}[U_\alpha g=\infty])$ has polar complement.

Now, we are going to solve the martingale problem for $\varphi_k$, so for $t\geq 0$ set
\begin{equation} \label{eq 3.4}
\beta_k(t):=\varphi_k(X(t))-\varphi_k(X(0))-\int_0^{t}L\varphi _k(X(s))ds
\end{equation}
and
\begin{equation} \label{eq 3.5}
M_k(t):=\varphi_k^{2}(X(t))-\varphi_k^{2}(X(0))-\int_0^{t}L\varphi_k^{2}(X(s))ds.
\end{equation}
Under $\mathbb{P}^{x}$, $x\in M_0\cap \widetilde{M}$, the integrals in (\ref{eq 3.4}) and (\ref{eq 3.5}) are well defined so that $\beta_k$ and $M_k$ are real-valued processes with continuous trajectories.
Indeed, the continuity of the integrals is due to the integrability ensured when $x\in M_0$, so it only remains to make sure that the integrals do not depend on the versions of $L\varphi_k$ resp. $L\varphi_k^{2}$. 
But this is a consequence of a more general statement: if $f:H_0\rightarrow \mathbb{R}$ is a measurable function s.t. $f=0 $ $\nu$-a.e., then $\int_0^{t}f(X(s))ds=0$ for all $t\geq 0$ $\mathbb{P}^{x}$-a.s. 
This is true because $f=0$ $\nu$-a.e. implies that $U_1 |f|=0$ $\nu$-a.e., hence $U_1|f|=0$ everywhere on $E$ by the strong Feller property (in fact, by a potential theoretical argument, knowing that $\nu$ is a reference measure is enough).

Let now $M:=M_0 \cap\widetilde{M}$, whose complement is clearly polar.
We claim that $\beta_k$ and $M_k$ are martingales w.r.t. $\mathbb{P}^{x}$ for all $x\in M$; in fact, the final aim is to show that $\beta_k, k\geq 1$ are independent standard Brownian motions. 
First, notice that on $M$ we have for all $\alpha >0, i \in {1,2}, k\geq 1$
$$
U_\alpha L\varphi_k^{i}=\alpha U_\alpha \varphi_k^{i} -\varphi_k^{i},
$$
since the sets $[|U_\alpha L\varphi_k^{i}-\alpha U_\alpha \varphi_k^{i} +\varphi_k^{i}|>0]\cap M$ are negligible (and hence of null potential) and finely open (w.r.t. $\mathcal{U}$), so they must be empty.

Let us show that $\beta_k$ is a martingale, the case of $M_k$ being similar.
For $0\leq s < t$, by the Markov property we get
$$\mathbb{E}^{x}\big [ \beta_k(t+s)-\beta_k(s) | \mathcal{F}_s\big ]= P_t\varphi_k(X(s))-\varphi_k(X(s))-\int_0^{t}P_rL\varphi_k (X(s))dr,$$
and because $E\setminus M$ is polar, it is enough to show that on $M$ we have
$$
P_t\varphi_k-\varphi_k-\int_0^{t}P_rL\varphi_k dr= 0.
$$
This is indeed true because on $M$ 
$$
\int_0^{\infty}e^{-\alpha t} (P_t\varphi_k-\varphi_k) dt=\frac{1}{\alpha}U_\alpha\varphi_k -\varphi_k=\frac{1}{\alpha}U_\alpha L\varphi_k=\int_0^{\infty}e^{{-\alpha t}}\int_0^{t}P_sL\varphi_k ds 
$$
for all $\alpha >0$. Since $t\mapsto \int_0^{t}P_sL\varphi_k ds $ is continuous, the claim follows by the uniqueness of the Laplace transform.

The proof now continues as in \cite{DaRoWa09}, and since we shall repeat most of the arguments therein used (but with certain adjustments) a little bit later, let us resume to recall briefly that the idea is to show that if $x\in M$ then $[\beta_k,\beta_{k'}]=t\delta_{k,k'}$ $\mathbb{P}^{x}$-a.s. so that
\begin{equation} \label{eq 3.6}
W:=\mathop{\sum}\limits_{k\geq 1}e_k\beta_k
\end{equation}
is a cylindrical Wiener process on $H$ under 
$\mathbb{P}^{x}$ and 
$(\Omega, \mathcal{F}, \mathcal{F}_t, , W, X, \mathbb{P}^{x})$ becomes a (weak) solution for equation (\ref{eq 3.1}) starting from $x$, in the sense of \cite{On04}, page 8; then, the monotonicity of $L$ entails $X$-pathwise uniqueness and the existence of a pathwise unique continuous strong solution follows by the Yamada-Watanabe type results from \cite{On04};  for more details see \cite{DaRoWa09}, page 20.

Let us rigourosly show that the solution on $M$ can be extended to a generalized solution on $E\setminus H_0$, in the sense of \cite{DaZa14}, Subsection 7.2.4. 
Let $(\widetilde{\Omega}, \widetilde{\mathcal{F}}, \widetilde{\mathcal{F}}_t ,  \widetilde{W}, \widetilde{\mathbb{P}})$ be a cylindrical Wiener process on $H$ and denote by $(X(t,x))_{t\geq 0}$ the pathwise unique continuous strong solution to equation (\ref{eq 3.1}) starting from $x\in M$, which is also Markov.
If $x\neq y \in M$, then as in \cite{DaRoWa09}, page 20,
$$
\langle e_k,X(t,x)-X(t,y)\rangle=\langle e_k,x-y \rangle - \lambda_k \int_0^{t}\langle e_k,X(s,x)-X(s,y)\rangle ds + \int_0^{t}\langle e_k,\widetilde{F}_0(X(s,x))-\widetilde{F}_0(X(s,y))\rangle ds .
$$
Then by the chain rule
\begin{align*}
\langle e_k,X(t,x)-X(t,y)\rangle^{2}=&\langle e_k,x-y \rangle^{2} - 2\lambda_k \int_0^{t}\langle e_k,X(s,x)-X(s,y)\rangle^{2} ds\\
 &+ 2\int_0^{t}\langle X(s,x)-X(s,y), e_k\rangle \langle e_k,\widetilde{F}_0(X(s,x))-\widetilde{F}_0(X(s,y))\rangle ds
\end{align*}
and summing up after $k$
\begin{align*}
|X(t,x)-X(t,y)|^{2}&\leq |x-y|^{2}+2\omega\int_0^{t}|X(s,x)-X(s,y)|^{2}ds\\
&\leq |x-y|^{2}e^{\omega t} \quad \mbox{ for all } t\geq 0
\end{align*}
where the second inequality follows by Gronwall inequality.
This means that the mapping
$$
H_0\supset M\ni x\mapsto X(\cdot,x) \in L^{\infty}(\Omega, C([0,T];H_0))
$$
is Lipschitz, and because $M$ is dense in $H_0$, it can be uniquely extended by continuity to the entire $H_0$, and such an extension is precisely a generalized solution in the sense of \cite{DaZa14}.

Let now $X(\cdot,x),\; x\in H_0\setminus M$ denote the extended process obtained in (i), and let us show it is a Markov process with transition function $(P_t)_{t\geq 0}$: if $f_1,...,f_n\in b\mathcal{B}$ continuous, $M\ni x_k \mathop{\to}\limits_k x \in H_0\setminus M$, and $0\leq t_1\leq \dots t_n<\infty$, then
\begin{align*}
\mathbb{E} \big [f_1(X(t_1,x))\cdots f_n(X(t_n,x)) \big ]&=\mathop{\lim}\limits_k \mathbb{E} \big [f_1(X(t_1,x_k))\cdots f_n(X(t_n,x_k))\big ]\\
&=\mathop{\lim}_k P_{t_1}f_1P_{t_2-t_1}f_2\cdots P_{t_n-t_{n-1}}f_n(x_k)\\
&=P_{t_1}f_1P_{t_2-t_1}f_2\cdots P_{t_n-t_{n-1}}f_n(x).
\end{align*}
By a monotone class argument, it follows that the Chapman-Kolmogorov identities are satisfied for all $f_1,...,f_n\in b\mathcal{B}$.

\vspace{0.2 cm}
\noindent
(ii). Similarly to the proof of Proposition \ref{prop 2.9}, we obtain that on the path space, the law $\mathbb{P}^{x}\circ X(\cdot)^{-1}$ is the same as $\overline{\mathbb{P}}\circ X(\cdot,x)^{-1}$, hence it is supported on the space $C([0, \infty); H_0)$ of all $|\cdot|$-continuous paths from $[0, \infty)$ to $H_0$.

\vspace{0.2 cm}
(ii.1). The assertion follows by Theorem \ref{thm 4.13} from Appendix.

\vspace{0.2 cm}
(ii.2). It follows precisely like \cite{DaRo02a}, Theorem 7.4, (ii), or as in  the proof of \thnameref{thm 3.16}, (ii)  below.

\vspace{0.2 cm}
(ii.3). In principle, the strategy is similar to the one for assertion (i), but this time we have to solve (\ref{eq 3.1}), with $B\equiv 0$, for initial distribution $\mathbb{P}^{x}\circ X(\varepsilon)^{-1}$, which is no longer a Dirac distribution. This comes with additional integrability issues, so let us give a rigorous proof:
let $x\in H_0\setminus M, \varepsilon > 0$ and set for all $t\geq 0$
$$\beta_k^{\varepsilon}(t):=\varphi_k(X(t+\varepsilon))-\varphi_k(X(\varepsilon))-\int_0^{t}L\varphi _k(X(s+\varepsilon))ds$$
and
$$M_k^{\varepsilon}(t):=\varphi_k^{2}(X(t+\varepsilon))-\varphi_k^{2}(X(\varepsilon))-\int_0^{t}L\varphi_k^{2}(X(s+\varepsilon))ds.$$

First of all, let us notice that the integrals appearing in the definitions of $\beta_k^{\varepsilon}$ and $M_k^{\varepsilon}$ are well defined. 
Indeed, by assumption $\widetilde{F}_0 \in L^{2}(\nu)$, hence $U_1|\widetilde{F}_0|^{2} \in L^{1}(\nu)$ and there exists $y\in H_0$ s.t. $U_1|\widetilde{F}_0|^{2}(y)<\infty$.
On the other hand, by Harnack inequality we have
$$
P_{s+\varepsilon}|\widetilde{F}_0|(x) \leq \Big ( P_{s+\varepsilon}|\widetilde{F}_0|^{2}(y)e^{\frac{2\omega |x-y|^{2}}{1-e^{-2\omega (s+\varepsilon)}}}\Big )^{\frac{1}{2}},
$$
which leads to
\begin{align*}
\int_0^{t} P_{s+\varepsilon}|\widetilde{F}_0|(x) ds & \leq \int_0^{t} \Big [P_{s+\varepsilon}|\widetilde{F}_0|^{2}(y)e^{\frac{2\omega |x-y|^{2}}{1-e^{-2\omega (s+\varepsilon)}}}\Big ]^{\frac{1}{2}} ds\\
&\leq \sqrt{t} e^{\frac{1}{2}(T+\varepsilon+\frac{2\omega |x-y|^{2}}{1-e^{-2\omega \varepsilon}})} \Big (\int_0^{t} e^{-(s+\varepsilon)} P_{s+\varepsilon}|\widetilde{F}_0|^{2}(y) ds\Big )^{\frac{1}{2}}\\
&\leq \sqrt{t} e^{\frac{1}{2}(T+\varepsilon+\frac{2\omega |x-y|^{2}}{1-e^{-2\omega \varepsilon}})} \Big( U_1|\widetilde{F}_0|^{2}(y)\Big )^{\frac{1}{2}}<\infty.
\end{align*}
In fact, since $L\varphi_k \in L^{2}(\nu)$, by a similar computation we get that $\beta_k^{x}(t) \in L^{1}(\mathbb{P}^{x})$ for all $t>0$. 
The fact that $M_k^{\varepsilon}$ are well defined can be done by localization, see below.

The fact that $\beta_k^{\varepsilon}$ is a martingale w.r.t. $\mathbb{P}^{x}$ follows as in (i) except that this time the polarity of $H_0\setminus M$ is crucial, namely
\begin{align*}
\mathbb{E}^{x}\Big [ \beta_k^{\varepsilon}(t+s)-\beta_k^{\varepsilon}(s) | \mathcal{F}_{s+\varepsilon} \Big ]
&=\mathbb{E}^{x}\Big [ \varphi_k(X(t+s+\varepsilon))-\varphi_k(X(s+\varepsilon)) - \int_s^{t+s}L\varphi_k(X(r+\varepsilon)) dr | \mathcal{F}_{s+\varepsilon} \Big ]\\
&= P_t\varphi_k(X(s+\varepsilon))-\varphi_k(X(s+\varepsilon))-\int_0^{t}P_rL\varphi_k (X(s+\varepsilon))dr\\
&=0,
\end{align*}
because $X(s+\varepsilon)\in M$ for all $s\geq 0$, and on $M$ we know from the proof of (i) that
$$
P_t\varphi_k-\varphi_k-\int_0^{t}P_rL\varphi_k dr=0.
$$

Now we show that $M_k^{\varepsilon}$ is a (continuous) local martingale w.r.t. $\mathbb{P}^{x}$, $x\in H_0\setminus M$.
Let $\theta_N\in C_0^{\infty}(\mathbb{R})$ such that $\theta_n(x)=x^{2}$ on $[-N,N]$.
Then $\theta_N\circ \varphi_k \in D(L)$ and
\begin{align*}
L\theta_N\circ\varphi_k(x)&=(\langle F_0(x),e_k \rangle+\langle x,Ae_k \rangle) \theta_N^{'}\circ\varphi_k(x)+\frac{1}{2} \theta_N^{''}\circ\varphi_k(x)\\
&=-\lambda_k\varphi_k(x)\theta_N^{'}\circ\varphi_k(x)+\langle \widetilde{F_0}(x),e_k \rangle\theta_N^{'}\circ\varphi_k(x)+\frac{1}{2} \theta_N^{''}\circ\varphi_k(x),
\end{align*}
in particular, $\theta_N\circ \varphi_k=\varphi_k^{2}$ and $L\theta_N\circ \varphi_k=L\varphi_k^{2}$ if $ |x|_H\leq N$ $\nu$-a.e.
Now, just as we did for $\beta_k^{\varepsilon}$, one can show that
$$
M_k^{\varepsilon, N}(t):=\theta_N\circ \varphi_k (X(t+\varepsilon))-\theta_N\circ \varphi_k (X(\varepsilon))-\int_0^{t}L \theta_N\circ \varphi_k (X(s+\varepsilon)) ds
$$
is an $(\mathcal{F}_{t+\varepsilon})_{t\geq 0}$-martingale.
Consequently, if we consider the $(\mathcal{F}_{t+\varepsilon})_{t\geq 0}$-stopping times $T_N:=\varepsilon + \inf \{ t>0 : |X(t+\varepsilon)|_H \geq N\} $, then $T_N\mathop{\nearrow}\limits_N \infty$ and by Doob's stopping theorem it follows that $M_k^{\varepsilon}(t\wedge T_N)=M_k^{\varepsilon, N}(t\wedge T_N), t\geq 0$ is a martingale.

Next, we show that $\beta_k^{\varepsilon}, k\geq 1$ are independent standard Brownian motions, and by Levy's characterization and polarization, it is sufficient to show that
$$
[\beta_k^{\varepsilon}, \beta_k^{\varepsilon}](t)=\int_0^{t}|D\varphi_k(X(s+\varepsilon))|^{2} ds\;=t.
$$
Since the second equality is trivial, let us show the first one.
\begin{align*}
[\beta_k^{\varepsilon}, \beta_k^{\varepsilon}](t)&=[\varphi_k(X(\cdot+\varepsilon)),\varphi_k(X(\cdot+\varepsilon))](t)-\varphi_k^{2}(X(\varepsilon))\\
&=\varphi_k^{2}(X(t+\varepsilon))-2\int_0^{t}\varphi_k(X(s+\varepsilon))d\varphi_k(X(s+\varepsilon))-\varphi_k^{2}(X(\varepsilon))\\
&=\int_0^{t}(L\varphi_k^{2}-2\varphi_k L \varphi_k)(X(s+\varepsilon))ds + M_k^{\varepsilon}(t)-2\int_0^{t}\varphi_k(X(s+\varepsilon))d\beta_k^{\varepsilon}(s)\\
&=\int_0^{t}|D\varphi_k(X(s+\varepsilon))|^{2} ds+ \mbox{ continuous local martingale.}
\end{align*}
The claim now follows since a continuous local martingale with paths of finite variation must be constant.

To conclude, since $W^{\varepsilon}:=\mathop{\sum}\limits_{k\geq 1}e_k\beta_k^{\varepsilon}$ defines a cylindrical Wiener process on $H$, we have that for $k\geq 1$
$$
\langle e_k,X(t+\varepsilon)\rangle=\langle e_k,X(\varepsilon)\rangle+\int_0^{t}\langle Ae_k,X(s+\varepsilon)\rangle + \langle e_k,F_0(X(s+\varepsilon))\rangle ds
+\langle e_k,W^{\varepsilon}(t)\rangle$$ for all $t\geq0$ $\mathbb{P}^{x}$-a.s.

By \cite{On04}, Theorem 13, it follows that 
$(\Omega,\mathcal{F}, \mathcal{F}_{t+\varepsilon},\mathbb{P}^{x}, W^{\varepsilon}, X(t+\varepsilon))$ is a solution to equation (\ref{eq 3.1})  in the sense of \cite{On04}, page 8.

\vspace{0.2 cm}
(iii) Let $x\in H_0\setminus M$, $\varepsilon>0$, and $\nu_\varepsilon^{x}:=\mathbb{P}^{x}\circ X(\varepsilon)^{-1}$.
Then by (ii.3), there exists a solution $X^{\varepsilon}$ (in the sense of \cite{On04}) for equation (\ref{eq 3.1}), with $B\equiv 0$, with initial distribution $\nu_\varepsilon^{x}$.
Then, the proof continues as in \cite{DaRoWa09}, page 20.
\end{proof}

\noindent
{\bf The case when $B \not\equiv 0$.}
Recall from Subsection 1.1 that for this case we fix a cylindrical Wiener process $\widetilde{W}$ on a stochastic basis $(\widetilde{\Omega}, \widetilde{\mathcal{F}}, \widetilde{\mathcal{F}}_t, \widetilde{\mathbb{P}})$, and take $(X(t,x))_{t\geq 0}$ to be the generalized solution given by (the proof of) \thnameref{thm 3.4}.

Before we prove \thnameref{thm 3.16} we need to go through several steps.
First we show that there exists a Markov process on $M$ with transition function $(Q_t)_{t\geq 0}$.

\begin{prop}\label{prop 3.5} There exists a right Markov process on $M$ with a.s. (norm) continuous paths and
with transition function $(Q_t)_{t \geq 0}$.
\end{prop}

\begin{proof}
Let $(\Omega,\mathcal{F}, \mathcal{F}_t,  X(t) ,\mathbb{P}^{x} )$ 
be the continuous Markov process on $M$ obtained in the first part of the proof of \thnameref{thm 3.4}, and let $W:=\mathop{\sum}\limits_{k\geq 1}e_k\beta_k$ be it's associated cylindrical Wiener process under $\mathbb{P}^{x}$ for each $x\in M$ (see (\ref{eq 3.4})), so that $(X, W, \mathbb{P}^{x})$ is a (continuous) solution to equation (\ref{eq 3.1}).
By construction, $W$ is an additive functional of $X$ (see (\ref{eq af}) below) in the sense of \cite{BlGe68} or \cite{Sh88}. 
Then
$$
M(t):=e^{\int_0^{t}\langle B(X(s))dW(s) \rangle-\frac{1}{2}\int_0^{t}|B|^{2}(X(s))ds}
$$
is a perfect martingale additive functional, hence
$$
\widetilde{Q}_tf(x):=\mathbb{E}^{x}\{f(X(t)) M(t)\}
$$
defines a Markovian semigroup which is the transition function of a right Markov process with continuous paths in $M$; for details, see \cite{Sh88}, Chapter VII, Section 62.

On the other hand, since on $M$ we have pathwise uniqueness for equation (\ref{eq 3.1}), with $B\equiv 0$, it follows  by \cite{On04} that joint uniqueness in law also holds for any initial distribution $\delta_x$, $x\in M$, so that $((X(t))_{t\geq0}, W, \mathbb{P}^{x})$ has the same law as $(((X(t,x))_{t\geq 0}, W, \widetilde{\mathbb{P}}))$.
Hence $\widetilde{Q}_t=Q_t$ for all $t\geq 0$ and the statement is proved.
\end{proof}

\begin{rem} \label{rem 3.7}
Concerning the existence of $Y$, we point out that the proof of the previous proposition doesn't work if $x\in H_0\setminus M$, because $W$ and hence $(M(t))_{t\geq 0}$ do not make sense under $\mathbb{P}^{x}$; in fact, this was the main technical difficulty which we had to deal with in the proof of \thnameref{thm 3.4}, because of the lack of integrability of $F_0(X(t))$ under $\mathbb{P}^{x}\otimes dt$ (see (\ref{eq 3.4}) and the subsequent remarks).
\end{rem}

\begin{lem} \label{lem 3.6} There exist a sequence of continuous functions $B_n:H_0\rightarrow H_0, n\geq 1$ such that $\lim_nB_n=B$ $\nu$-a.e., and $|B_n|_\infty \leq |B|_\infty$.

Further, if $Q_t^{n}$ and $\rho_t^{n,x}$ are given by (\ref{eq 3.7}) resp. (\ref{eq 3.8}) with $B$ replaced by $B_n$, then 

\begin{equation} \label{eq 3.9}
\mathbb{E}\{(\rho_t^{n,x})^{2}\}\leq e^{t|B|_\infty^{2}}, n\geq 1
\end{equation}
\begin{equation} \label{eq 3.10}
\lim\limits_n\rho_t^{n,x}=\rho_t^{x} \mbox{ in } L^{1}(\widetilde{\mathbb{P}}),
\end{equation}
\begin{equation} \label{eq 3.11}
\lim\limits_k\rho_t^{n,x_k}=\rho_t^{n,x} \mbox{ in } L^{1}(\widetilde{\mathbb{P}}), n\geq 1
\end{equation}
for all $H_0 \ni (x_k)_{k\geq 1} \mathop{\to}\limits_k x$ in $H_0$.
In particular, $Q_t^{n}$ is Feller and
\begin{equation} \label{eq 3.12}
\lim\limits_nQ_t^{n}f=Q_tf \mbox{ pointwise on } H_0.
\end{equation}
\end{lem}

\begin{proof}
To prove the first statement, let $(e_k)_{k\geq 1}$ be an orthonormal basis of $H$
and set $B_{k,\alpha}(x):=\sum\limits_{i=1}^{k}\alpha U_\alpha(\langle B(\cdot), e_i\rangle)(x) e_i$ for all $x\in H_0$, $k\geq 1$ and $\alpha >0$. 
Clearly, for each $k\geq 1$ we have that $\lim\limits_{k\to \infty}\lim\limits_{\alpha \to \infty}B_{k,\alpha} = B$ in $L^{2}(\nu)$, hence, by a diagonal argument, there exists a subsequence $(\alpha(k))_{k\geq 1}$ such that $B_k:=B_{k,\alpha(k)}$ converges to $B$ $\nu$-a.e. 
Moreover, each $B_k$ is continuous by the strong Feller property of $\mathcal{U}$, and
\begin{align*}
|B_k(x)|^{2} &= \sum\limits_{i=1}^{k}|\alpha(k) U_{\alpha(k)}(\langle B(\cdot), e_i\rangle)(x)|^{2} \leq\alpha(k) U_{\alpha(k)}\Big( \sum\limits_{i=1}^{k}|\langle B(\cdot), e_i\rangle|^{2}\Big)(x)\\
& \leq \alpha(k) U_{\alpha(k)}\Big(|B(\cdot)|^{2}\Big)(x) \leq |B|_\infty^{2}.
\end{align*}

Further, by $d\rho_t^{n,x}=\langle\rho_t^{n,x}B_{n}(X(t,x)), dW(t)\rangle$ and It\^o isometry we have
$$
\mathbb{E}\{(\rho_t^{n,x})^{2}\}=1+\int_0^{t}\mathbb{E}\{(\rho_s^{n,x})^{2} |B_{n}(X(s,x))|^{2} \} ds \leq 1+|B|_\infty^{2} \int_0^{t}\mathbb{E}\{(\rho_s^{n,x})^{2} \} ds,
$$
so the estimate in (\ref{eq 3.9}) follows by Gronwall's lemma.

We show now at the same time that both (\ref{eq 3.10}) and (\ref{eq 3.12}) hold, pointing out in advance that below we use $|e^{a}-e^{b}|\leq (e^{a}+e^{b})|a-b|$ for the second inequality, the triangle inequality and It\^o isometry for the third inequality, and the bound in (\ref{eq 3.9}) and Cauchy-Schwarz inequality for the fourth one:
\begin{align*}
|Q_t&^{n}f(x)-Q_tf(x)| \leq |f|_\infty \mathbb{E}\Big\{|\rho_t^{n,x}-\rho_t^{x}|\Big\}\\
&\leq |f|_\infty \mathbb{E}\Big\{(\rho_t^{n,x}+\rho_t^{x})\Big|\int_0^{t}\langle (B_{n}-B)(X(s,x)) dW_s \rangle-\frac{1}{2}\int_0^{t}(|B_n|^{2}-|B|^{2})(X(s,x))ds\Big|\Big\}\\
& \leq  |f|_\infty |\rho_t^{n,x}+\rho_t^{x}|_{L^{2}(\mathbb{P})}\Big(\mathbb{E}\Big\{\int_0^{t}|B_{n}-B|^{2}(X(s,x)) ds\Big\}^{\frac{1}{2}}+\frac{1}{2}\Big|\int_0^{t}(|B_n|^{2}-|B|^{2})(X(s,x))ds\Big|_{L^{2}(\widetilde{\mathbb{P}})}\Big)\\
& \leq 2 |f|_\infty e^{\frac{t}{2}|B|_\infty^{2}} \Big(\mathbb{E}\Big\{\int_0^{t}|B_{n}-B|^{2}(X(s,x)) ds\Big\}^{\frac{1}{2}}+\sqrt{t}|B|_\infty\mathbb{E}\Big\{\int_0^{t}|B_{n}-B|^{2}(X(s,x)) ds\Big\}^{\frac{1}{2}}\Big)\\
& = 2 |f|_\infty e^{\frac{t}{2}|B|_\infty^{2}} (1+\sqrt{t}) \mathbb{E}\Big\{\int_0^{t}|B_{n}-B|^{2}(X(s,x)) ds\Big\}^{\frac{1}{2}}\\
&\leq 2 |f|_\infty e^{\frac{t}{2}|B|_\infty^{2}} (1+\sqrt{t}) e^{\frac{t}{2}}(U_1(|B_{n}-B|^{2})(x))^{\frac{1}{2}}.
\end{align*}
Hence (\ref{eq 3.10}) and (\ref{eq 3.12}) follow by dominated convergence, due to the first part of Lemma \ref{lem 3.6} and the fact that $\nu$ is a reference measure.

To prove (\ref{eq 3.11}), let $x_k\mathop{\rightarrow}\limits_k x$ in $H_0$ and fix $n \geq 1$. 
As above, we get
\begin{align*}
\mathbb{E}\Big\{|\rho_t^{n,x_k}-\rho_t^{n, x}|\Big\}
&\leq |f|_\infty \mathbb{E}\Big\{(\rho_t^{n,x_k}+\rho_t^{n,x})\Big|\int_0^{t}\langle- (B_{n}(X(s,x_k)) -B_{n}(X(s,x)) )dW(s) \rangle\\
&\quad -\frac{1}{2}\int_0^{t}(|B_n|^{2}(X(s,x_k)) -|B_n|^{2}(X(s,x)))ds\Big|\Big\}\\
&\leq 2 |f|_\infty e^{\frac{t}{2}|B|_\infty^{2}} (1+\sqrt{t}) \mathbb{E}\Big\{\int_0^{t}|B_{n}(X(s,x_k)) -B_{n}(X(s,x))|^{2} ds\Big\}^{\frac{1}{2}}\\
\end{align*}
 and the convergence follows by dominated convergence.

To show that $Q_t^{n}$ is Feller, let $f$ be bounded and continuous, and $x_k\mathop{\rightarrow}\limits_k x$ in $H_0$.
Then
\begin{align*}
|Q_t^{n}f(x_k)-Q_t^{n}f(x)| &\leq |f|_\infty \mathbb{E}\Big\{|\rho_t^{n,x_k}-\rho_t^{n,x}|\Big\}+\mathbb{E}\Big\{|f(X(t,x_k))-f(X(t,x))|\rho_t^{n,x}\Big\}\\
& \leq |f|_\infty \mathbb{E}\Big\{|\rho_t^{n,x_k}-\rho_t^{n,x}|\Big\}+\mathbb{E}\Big\{|f(X(t,x_k))-f(X(t,x))|^{2}\Big\}^{\frac{1}{2}}\mathbb{E}\Big\{(\rho_t^{n,x})^{2}\Big\}^{\frac{1}{2}}
\end{align*}
which converges to $0$ when $k$ tends to $\infty$, by bounded convergence, (\ref{eq 3.9}) and (\ref{eq 3.11})
\end{proof}

\begin{prop} \label{prop 2.4}
$(Q_t)_{t\geq 0}$ is a Markov transition function on $H_0$.
\end{prop}

\begin{proof}
Let $B_n$ and $Q_t^{n}, n\geq 1$ be as in Lemma \ref{lem 3.6}.
 By Proposition \ref{prop 3.5} applied to $B_n$, for each $x \in M$
\begin{equation} \label{eq 3.13}
Q^{n}_{t+s}f(x)=\mathbb{E}\{\rho_t^{x}Q^{n}_sf(X(t,x))\}.
\end{equation}
To extend (\ref{eq 3.13}) to the entire $H_0$, let $M\ni x_k\mathop{\to}\limits_k x\in H_0\setminus M$ and $f$ be bounded and continuous. Then
$$
Q^{n}_{t+s}f(x)=\lim\limits_k Q^{n}_{t+s}f(x_k)=\lim\limits_k \mathbb{E}\{\rho_t^{n,x_k}Q^{n}_sf(X(t,x_k))\}=\mathbb{E}\{\rho_t^{n,x}Q^{n}_sf(X(t,x_k))\}=Q^{n}_tQ^{n}_sf(x),
$$ 
where the first and the third equalities follow by the Feller property of $Q^{n}_t$ and by (\ref{eq 3.11}) from Lemma \ref{lem 3.6}.

To conclude, if $x\in H_0$ then
$$
Q_{t+s}f(x)=\lim\limits_n Q^{n}_{t+s}f(x)=\lim\limits_n \mathbb{E}\{\rho_t^{n,x}Q^{n}_sf(X(t,x))\}=\mathbb{E}\{\rho_t^{x}Q_sf(X(t,x))\}=Q_tQ_sf(x),
$$
where for the first and the third identities we used Lemma \ref{lem 3.6}, (\ref{eq 3.10}) and (\ref{eq 3.12}).
The case of general $f$ follows now by a monotone class argument.


\end{proof}

We crucially need the following It\^o formula which is a refinement of \cite{DaFlRoVe16}, Corollary 3.14, in the sense that we can prove it for all starting points $x\in M$, there is no need to exclude a further $\nu$-inessential set. Also, our proof is a bit simpler and uses more general arguments.
 
First of all, we recall that by \cite{DaFlRoVe16}, Lemma 3.6, and Remark \ref{rem 3.3}, any Lipschitz function $\varphi:H\rightarrow \mathbb{R}$ is Gateaux differentiable $\nu$-a.e.; we denote its derivative by $\nabla \varphi$.
\begin{prop} \label{prop 3.8} Let $X$ be the right Markov process and $M$ the set with polar complement provided by \thnameref{thm 3.4}.
Also, let $f=U_\alpha g$ for some $\alpha >0$ and $g\in b\mathcal{B}(H_0)$.
Then the following It\^o formula holds $\mathbb{P}^{x}$-a.s. for all $x\in M$:
\begin{equation}\label{eq 3.14}
f(X(t))-f(x)-\int_0^{t}Lf(X(s)) ds=\int_0^{t}\langle \nabla f(X(s)), dW(s)\rangle \quad t\geq 0.
\end{equation}
\end{prop}

\begin{proof}
Note that the two sides of (\ref{eq 3.14}) are continuous martingales by \thnameref{thm 3.4} (and it's proof) under $\mathbb{P}^{x}$ for all $x\in M$, and also that the rhs does not depend on the $\nu$-version of $\nabla f$ due to the It\^o isometry and the fact that $\mathcal{U}$ is strong Feller.

First, we show that (\ref{eq 3.14}) holds almost surely w.r.t. $\mathbb{P}^{\nu}:=\int \mathbb{P}^{x} \; \nu(dx)$, and note that by the continuity of the paths, it is enough to prove that (\ref{eq 3.14}) holds a.s. for each $t\geq 0$.

Since $\mathcal{F}\mathcal{C}_b^{2}$ is a core for $(L, D(L))$, there exists a sequence $(f_n)_{n\geq 1}\subset \mathcal{F}\mathcal{C}_b^{2}$ s.t. 
\begin{equation} \label{eq 3.150}
\lim\limits_n\int_H [(f-f_n)^{2} + |\nabla (f-f_n)|^{2} + (L(f-f_n))^{2}] \; d\nu =0.
\end{equation} 
Also, since $X$ is a weak solution for (\ref{eq 3.1}), with $B\equiv 0$, and by classical It\^o's formula in finite dimensions, we get $\mathbb{P}^{x}$-a.s., $x\in M$ that
\begin{equation} \label{eq 3.160}
f_n(X(t))-f_n(x)-\int_0^{t}Lf_n(X(s)) ds=\int_0^{t}\langle \nabla f_n(X(s)), dW(s)\rangle, \; t\geq 0.
\end{equation}
Now, by It\^o isometry and the fact that $\nu$ is $(P_t)$-invariant, it is straightforward to see that due to (\ref{eq 3.150}) we can pass to the limit in (\ref{eq 3.160}) and obtain that (\ref{eq 3.14}) holds $\mathbb{P}^{\nu}$-a.s.

Let us now show that the formula holds $\mathbb{P}^{x}$-a.s. for all $x\in M$.
By definition, $W$ is an additive functional of $X$ (see (\ref{eq af}) below), hence so does the rhs of (\ref{eq 3.14}), and the same is true for the lhs.
Therefore, by substracting the two sides of (\ref{eq 3.14}), we find ourselves in the following abstract situation: we have a real valued continuous martingale starting from $0$, say $(N(t))_{t\geq 0}$, which is an additive functional of $X$, satisfies $N(t)=0$ for all $t\geq 0$ $\mathbb{P}^{\nu}$-a.s., and $\mathbb{E}^{x}(N(t)^{2})\leq const \cdot t^{2}$.
We claim that $(N(t))_{t\geq 0}$ is the null process $\mathbb{P}^{x}$-a.s. for all $x\in M$.
To this end, define for all $x\in M$ the finite measurable function
\begin{equation*}
v(x):=\mathbb{E}^{x}\Big\{\int_0^{\infty}e^{-t}N(t)^{2} dt\Big\}.
\end{equation*}
Let us show that v is $1$-excessive: for all $s\geq 0$ and $x\in M$ we have
\begin{align*}
P_sv(x)&= \mathbb{E}^{x}\Big\{\mathbb{E}^{X(s)}\Big\{\int_0^{\infty}e^{-t}N(t)^{2} dt\Big\}\Big\}=\mathbb{E}^{x}\Big\{\int_0^{\infty}e^{-t}[N(t)\circ\theta(s)]^{2} dt\Big\}\\
&=\mathbb{E}^{x}\Big\{\int_0^{\infty}e^{-t}[N(t+s)-N(s)]^{2} dt\Big\}=\int_0^{\infty}e^{-t}\Big\{\mathbb{E}^{x}[N(t+s)^{2}]-\mathbb{E}^{x}[N(s)^{2}]\Big\} dt\\
&\leq \int_0^{\infty}e^{-t}\mathbb{E}^{x}[N(t+s)^{2}] dt \leq \int_s^{\infty}e^{-(t-s)}\mathbb{E}^{x}[N(t)^{2}] dt\\
& \leq e^{s} v(x),
\end{align*}
which proves that $v$ is $1$-supermedian.
Also, by dominated convergence, the third equality from above also gives $P_sv\mathop{\rightarrow}\limits_{s\to 0} v$ pointwise on $M$, which proves that $v$ is indeed $1$-excessive.

To conclude, we have that the set $[v>0]\cap M$ is $\nu$-negligible and finely open, and since $\nu$ is a reference measure, it is well known that $[v>0]\cap M =\emptyset$.

\end{proof}

\begin{rem} \label{rem 3.12}
By the pathwise uniqueness (hence joint uniqueness in law) which holds on $M$, It\^o's formula (\ref{eq 3.14}) holds true for any solution of equation (\ref{eq 3.1}), with $B\equiv 0$, in particular for $((X(t,x))_{t\geq 0}, \widetilde{\mathbb{P}})$ for all $x\in M$.
\end{rem}

For each $n\geq 1$ we denote by $(V_\alpha^{n})_{\alpha>0}$ the resolvent associated to $(Q_t^{n})_{t\geq 0}$ given in Lemma \ref{lem 3.6}.

\begin{lem} \label{lem 3.9}
The resolvent $(V_\alpha^{n})_{\alpha>0}$ has the Feller property, and if $\alpha > \frac{|B|_\infty}{2}$, then for all $f\in b\mathcal{B}$
\begin{equation} \label{eq 3.16}
 \lim\limits_n V_\alpha^{n}f=V_\alpha f \quad \mbox{pointwise on } H_0.
\end{equation}
\end{lem}
\begin{proof} Since $Q_t^{n}, t\geq 0$ is Feller by Lemma \ref{lem 3.6}, the first part of the statement is clear by dominated convergence.

Now, if $x\in H_0$ then
\begin{align*}
|V_\alpha^{n}f-V_\alpha f|(x) &\leq \mathbb{E}\Big\{\int_0^{\infty}e^{-\alpha t} f(X(t,x))|\rho_t^{n,x}-\rho_t^{x}| \; dt\Big \}\\
& \leq  |f|_\infty \int_0^{\infty}e^{-\alpha t}\mathbb{E}\{|\rho_t^{n,x}-\rho_t^{x}| \} \; dt,
\end{align*}
which converges to $0$ when $n$ tends to $\infty$ by dominated convergence, after using (\ref{eq 3.9}) and (\ref{eq 3.10}) from Lemma \ref{lem 3.6}. 

\end{proof}

\begin{proof}[Proof of Theorem \ref{thm 3.13}]
Let $f\in b\mathcal{B}$ and $\varphi=U_\alpha f$, for some fixed $\alpha\geq 4\pi|B|_\infty^{2}$.
First, let $x\in M$.
Since $d\rho_t^{x}=\langle\rho_t^{x}B(X(t,x)), dW(t)\rangle$ and $\varphi (X(t,x))$ satisfies the It\^o formula (\ref{eq 3.14}) (see Remark \ref{rem 3.12}), the integration by parts formula gives $\mathbb{P}$-a.s. for all $t$
\begin{equation} \label{eq 3.19}
\rho_t^{x}\varphi (X(t,x))= \varphi (x) + \int_0^{t}\rho_s^{x}[L\varphi(X(s,x)) +\langle B(X(s,x)), \nabla \varphi (X(s,x)) \rangle] \; ds + N(t),
\end{equation} 
where $(N(t))_{t\geq 0}$ is a continuous martingale which by Lemma \ref{lem 3.6}, (\ref{eq 3.9}) satisfies 
\begin{equation} \label{eq 3.20}
\mathbb{E}|N(t)| \leq |\varphi|_\infty e^{\frac{t}{2}|B|_\infty^{2}}.
\end{equation}
Using (\ref{eq 3.19}) and (\ref{eq 3.20}), we have
\begin{align*}
V_\alpha \varphi(x)&=\mathbb{E}\Big\{ \int_0^{\infty}e^{-\alpha t}\rho_t^{x}\varphi (X(t,x)) \; dt\Big\}\\
&=\mathbb{E}\Big\{ \int_0^{\infty}e^{-\alpha t}\Big[\varphi(x)+\int_0^{t}\rho_s^{x}[L\varphi(X(s,x)) +\langle B(X(s,x)), \nabla \varphi (X(s,x)) \rangle] \; ds \; \Big]dt\Big\}\\
&=\frac{1}{\alpha} \varphi (x) +  \int_0^{\infty}e^{-\alpha t}\int_0^{t}Q_s[L\varphi +\langle B, \nabla \varphi \rangle](x) \; ds \; dt\\
&=\frac{1}{\alpha} \varphi (x) + \frac{1}{\alpha} \int_0^{\infty}e^{-\alpha t}Q_t(L\varphi +\langle B, \nabla \varphi \rangle)(x) \; dt\\
&=\frac{1}{\alpha} \varphi (x) + \frac{1}{\alpha} V_\alpha(L\varphi +\langle B, \nabla \varphi \rangle)(x)\\
&=\frac{1}{\alpha} \varphi (x) + \frac{1}{\alpha} V_\alpha(L\varphi +\langle B, \nabla \varphi \rangle)(x),
\end{align*}
and since $\varphi=U_\alpha f$, this means that $V_\alpha(f-\langle B, \nabla U_\alpha f \rangle)=U_\alpha f$ on $M$.


The final step of the proof is to extend (\ref{eq 3.18}) from $M$ to the entire $H_0$.
To this end, it is sufficient (and, of course, necessary) to show that $\mathcal{V}$ is Lipschitz strong Feller.
We first prove this for $V_\alpha^{n}$, which a priori has the Feller property by Lemma \ref{lem 3.9}. Also, it is clear that the above computations hold true for each $B_n,n\geq1$,
hence, if $f$ is continuous and bounded, then
\begin{equation}
V^{n}_\alpha f=U_\alpha(I-\langle B^{n}, \nabla U_\alpha \rangle)^{-1}f \quad \mbox{ on } M,
\end{equation}
and by the Feller property of $V_\alpha^{n}$ and the strong Feller property of $U_\alpha$, the previous identity hold on the entire $H_0$.
But now, by (\ref{estimate}) we have that $|V_\alpha^{n}f|_{\rm Lip}\leq 2\sqrt{\frac{\pi}{\alpha}}|f|_\infty$, hence by Lemma \ref{lem 3.9} the same estimate holds true for $V_\alpha f$.
If $f$ is merely bounded, then $|V_\alpha f|_{\rm Lip}\leq \sup\limits_k |V_\alpha (kU_k f)|_{\rm Lip} \leq 2\sqrt{\frac{\pi}{\alpha}}|kU_kf|_\infty\leq 2\sqrt{\frac{\pi}{\alpha}}|f|_\infty$, and the proof is complete.
\end{proof}

\begin{proof}[Proof of \thnameref{thm 3.16}] 
In order to prove the existence of the process, the idea is to apply Theorem \ref{thm 2.2}. 
First of all, by Proposition \ref{prop 3.5}, there exists a subset $M\in \mathcal{B}(H_0)$ and a right process on $M$ with a.s. (norm) continuous paths and resolvent $\mathcal{V}|_M$, the restriction to $M$ of $\mathcal{V}$.
It is clear that $\mathcal{V}$ is an extension of $\mathcal{V}|_M$ from $M$ to $H_0$. 
Also, condition (H) from Section 2 is fulfilled for e.g. $\mathcal{C}:=\{f\in C_b(H_0) : f\geq 0\}$, taking into account the structure of $(Q_t)_{t\geq 0}$ given by (\ref{eq 3.7}), and the Lipschitz strong Feller property obtained in Theorem \ref{thm 4.15}.
Hence, by Theorem \ref{thm 2.2}, there exists a right process $Y = (Y(t),  \mathbb{Q}^{x})$ on $H_0$ with resolvent $\mathcal{V}$, and by the Lipschitz strong Feller property of $\mathcal{V}$, we can apply Proposition \ref{prop 2.10} to make sure that there exists a natural topology $\tau_0$ on $H_0$ which is coarser than the norm topology. 
Consequently, using also Proposition \ref{prop 2.9}, (iii), $Y$ has a.s. $\tau_0$-continuous paths.
Also, note that if $T>0$, the space $C([0,T], H_0)$ of $|\cdot|$-continuous paths from $[0,T]$ to $H_0$ is a measurable subset of the space $C_{\tau_0}([0,T], H_0)$ of $\tau_0$-continuous paths from $[0,T]$ to $H_0$, endowed with the $\sigma$-algebra generated by the canonical projections. 
Now, because $(\rho_t^{x})_{t\geq 0}$ is a martingale under $\mathbb{P}$ for all $x\in H_0$, it follows that the law under $\mathbb{Q}^{x}$ of $(Y(t))_{t\in [0,T]}$ on $C_{\tau_0}([0,T], H_0)$ is precisely the law of $(X(t)^{x})_{t\in [0,T]}$ on $C_{\tau_0}([0,T], H_0)$ under $\rho_T^{x}\cdot\widetilde{\mathbb{P}}$.
But the latter is supported on $C([0,T], H_0)$, hence so is the law of $(Y(t))_{t\in [0,T]}$.

Let us turn now to assertions (i-iv):

(i). By Lemma \ref{lem 3.6}, the Cauchy-Schwartz inequality, and the fact that $\nu$ is $P_t$-invariant, we have
$$
\int (Q_tf)^{2} d\nu \leq e^{2t|B|_\infty^{2}} \int P_t(f^{2})d\nu =  e^{2t|B|_\infty^{2}} \int f^{2} d\nu \quad \mbox{ for all } f\in b\mathcal{B}(H_0).
$$ 
This means that $(Q_t)_{t\geq 0}$ can be extended to a semigroup of Markov operators on $L^{2}(\nu)$, with corresponding norms less than $ e^{t|B|_\infty^{2}}$.
On the other hand, $\lim\limits_{t\to 0}Q_tf=f$ for all $f\in C_b(H_0)$, and one can easily extend this convergence for all $f\in L^{2}(\nu)$ by density.
Hence $(Q_t)_{t\geq 0}$ is a $C_0$-semigroup of quasi-contractions on $L^{2}(\nu)$, and let $(L^{B},D(L^{B}))$ be it's infinitesimal generator.
Now, if $x\in M$ then $(X(t,x))_{t\in [0,T]}$ is a solution for (\ref{eq 3.1}), with $B\equiv 0$, and as already mentioned in Remark \ref{rem 3.6}, by Girsanov transformation we we have that equation (\ref{eq 3.1}), case $B\not\equiv 0$, has a weak solution given by $(X(t,x))_{t\in [0,T]}$, but under $\rho_T^{x}\cdot \mathbb{P}$. 
By applying It\^o's formula for $\varphi\in \mathcal{E}_A(H)$ and then taking expectations, we get that on $L^{2}(\nu)$
\begin{equation} \label{eq 3.23}
Q_t\varphi =\varphi + \int_0^{t} Q_s(L\varphi+\langle B, D \varphi\rangle) ds,
\end{equation}
hence 
\begin{equation} \label{eq 3.24}
L^{B}\varphi=\lim\limits_{t\to 0}\frac{Q_t\varphi-\varphi}{t}=L\varphi+\langle B, D \varphi\rangle \quad \mbox{ in } L^{2}(\nu).
\end{equation}
Now, let $\alpha>|B|^{2}_\infty$ and note that $V_\alpha (b\mathcal{B}(H_0))$ is a core for $(L^{B}, D(L^{B}))$.
On the other hand, by (\ref{eq 3.18}) we have $V_\alpha (b\mathcal{B}(H_0))=U_\alpha (b\mathcal{B}(H_0))$.
Therefore, since $\mathcal{E}_A(H)$ is a core for $(L, D(L))$, it is dense in $V_\alpha (b\mathcal{B}(H_0))$ in the graph norm w.r.t. $L^{B}$, and this concludes the proof.

(ii) Since we do not know that $Q_t$ is (strong) Feller, the proof of  \cite{DaRo02a}, Theorem 7.4, (ii) requires a slight modification, so let us give a complete rigorous proof. First of all, note that $D_0=U_\alpha (b\mathcal{B}(H_0))=V_\alpha (b\mathcal{B}(H_0))$ for one (hence all) $\alpha >0$.
Then, the integral appearing in (\ref{eq 3.22}) is well defined, i.e. are finite (because $\varphi \in V_\alpha (b\mathcal{B}(H_0))$) and do not depend on $\nu$-classes because of the general argument invoked after relation (\ref{eq 3.5}).
Hence, the process $A(t):=\varphi (Y(t))-\varphi (Y(0))-\int_{0}^{t}L^{B}\varphi (Y(s))ds, t\geq 0$ is an (bounded) additive functional, i.e. 
\begin{align}  \label{eq af}
A(t+s)=A(t)+A(s)\circ \theta(s), t,s\geq 0 \mbox{  a.s.,}
 \end{align} 
and it is well known that such an additive functional is a martingale if and only if it has zero expectation, hence if and only if
\begin{equation} \label{eq 3.25}
Q_t\varphi(x)=\varphi(x)+\int_0^{t}Q_s(L^{B}\varphi)(x)ds \quad \mbox{ for all }  t\geq 0, x\in H_0.
\end{equation}
However, (\ref{eq 3.25}) holds true $\nu$-a.e. by (i), and because $\varphi=V_\alpha g$ for some $g\in b\mathcal{B}$, $\alpha>0$, and $\mathcal{V}$ is strong Feller, it follows that $Q_t\varphi=V_\alpha Q_t g$ is continuous.
Therefore, by density, it is sufficient to show that the rhs of (\ref{eq 3.25}) is continuous in $x\in H_0$: if $ x_n\mathop{\to}\limits_n x$ in $H_0$, then
\begin{align*}
\big|\int_0^{t}&Q_s(L^{B}\varphi)(x_n)-Q_s(L^{B}\varphi)(x)\big|\leq\big|\int_0^{t}Q_s(L^{B}\varphi)(x_n)-e^{-\alpha s}Q_s(L^{B}\varphi)(x_n)\big|+\\
&+\big|\int_0^{t}Q_s(L^{B}\varphi)(x)-e^{-\alpha s}Q_s(L^{B}\varphi)(x)\big|+\big|\int_0^{t}e^{-\alpha s}[Q_s(L^{B}\varphi)(x_n)-Q_s(L^{B}\varphi)(x)]\big|\\
&\leq 2|L^{B}\varphi|_\infty \int_0^{t} (1-e^{-\alpha s}) ds + \big|\int_0^{t}e^{-\alpha s}[Q_s(L^{B}\varphi)(x_n)-Q_s(L^{B}\varphi)(x)]\big|\\
&= 2|L^{B}\varphi|_\infty \int_0^{t} (1-e^{-\alpha s}) ds + \\
&+\big| e^{-\alpha s}[ V_\alpha Q_t(L^{B}\varphi)(x_n)-V_\alpha Q_t(L^{B}\varphi)(x)] - [V_\alpha (L^{B}\varphi)(x_n)- V_\alpha (L^{B}\varphi)(x)] \big|,
\end{align*}
where for the last equality we used the formula $e^{-\alpha t}V_\alpha Q_t f=V_\alpha f-\int_0^{t}e^{-\alpha s} Q_sf ds$ for all bounded $f$.
Now, using the strong Feller property of $\mathcal{V}$, let $n$ tend to infinity and then $\alpha$ to $0$ in order to obtain the claim.

(iii). This assertion follows similarly to the weak existence part from the proof of \thnameref{thm 3.4}, (i), with the mention that  $f(\varphi_{k}) \in D(L^{B})$ and $L^{B}(f(\varphi_{k}))=L_0^{B}(f(\varphi_{k}))$ for all twice differentiable functions $f$ with continuous derivatives, since (\ref{eq 3.23}) and (\ref{eq 3.24}) remain true with the same justification.

(iv). The proof of the last statement is similar to the proof of \thnameref{thm 3.4}, (ii.3), once we show that $\int_0^{t}Q_{s+\varepsilon}|\widetilde{F}_0|(x) ds < \infty$ for all $t, \varepsilon >0$ and $x\in H_0 \setminus M$.
By \cite{Pr05}, Theorem 39, we have that $|\rho_t^{x}|_{L^{q}(\mathbb{P})}\leq e^{qt|B|^{2}_\infty}$ for all $1\leq q<\infty$, hence by Holder inequality we get
$Q_{s+\varepsilon}|\widetilde{F}_0|(x)\leq e^{q^{\ast}t|B|^{2}_\infty}(P_{s+\varepsilon}|\widetilde{F}_0|^{q})^{\frac{1}{q}}(x)$ for all $1<q<\infty, \frac{1}{q}+\frac{1}{q^{\ast}}=1$. This means that for $q$ small enough,  $|\widetilde{F}_0|^{q}\in L^{p}(\nu)$ for some $p>1$, and we can conclude just like in the proof of \thnameref{thm 3.4}, (ii.3), using the Harnack inequality for $(P_t)_{t\geq 0}$.
\end{proof}

\section{Appendix: a brief overview of right processes}
Throughout this section we follow mainly the terminology of \cite{BeBo04a}, but we also heavily refer to the classical works \cite{BlGe68}, \cite{Sh88}; see also the references therein. 

Let $(E, \mathcal{B})$ be a Lusin measurable space. 
We denote by $(b)p\mathcal{B}$ the set of all numerical, (bounded) positive $\mathcal{B}$-measurable functions on $E$.
Throughout, by $\mathcal{U}=(U_\alpha)_{\alpha>0}$ we denote a resolvent family of (sub-)Markovian of kernels on $(E, \mathcal{B})$.
If $\beta>0$, we set $\mathcal{U}_\beta:=(U_{\alpha+\beta})_{\alpha>0}$.

\begin{defi} \label{defi 4.1}
A $\mathcal{B}$-measurable function $v:E\rightarrow \overline{\mathbb{R}}_+$ is called excessive (w.r.t. $\mathcal{U}$) if $\alpha U_\alpha v \leq v$ for all $\alpha >0$ and $\mathop{\sup}\limits_\alpha \alpha U_\alpha v =v $ point-wise; by $\mathcal{E(\mathcal{U})}$ we denote the convex cone of all excessive functions w.r.t. $\mathcal{U}$.
\end{defi}
If a $\mathcal{B}$-measurable function $w: E \rightarrow \overline{\mathbb{R}}_+$ is merely $\mathcal{U}_\beta$-supermedian (i.e. $\alpha U_{\alpha+\beta} w \leq w$ for all $\alpha > 0$), then its $\mathcal{U}_\beta${\it -excessive regularization} $\widehat{w} \in \mathcal{E(U)}$ is defined as
$\widehat{w}:=\mathop{\sup}\limits_\alpha \alpha U_{\beta +\alpha}w$.

\noindent
{\bf (H')} Throughout this section we assume that $\mathcal{E}(\mathcal{U}_\beta)$ is min stable, contains the constant functions, and generates $\mathcal{B}$ for one (hence all) $\beta >0$. 

\begin{rem} \label{rem 4.2}
\begin{enumerate}[(i)]
\item (H') is a minimal requirement, which is automatically satisfied if there exists a right process associated to $\mathcal{U}$; see Definition \ref{defi 4.4} below.
\item By Corollary 2.3 from \cite{BeRo11a}, (H') is equivalent with the existence of a vector lattice $\mathcal{C}\subset b\mathcal{B}$ satisfying (i)-(iii) of our main condition (H) from the beginning of Section 2.
\end{enumerate}
\end{rem}

\begin{defi} \label{defi 3.3}
\begin{enumerate}[(i)]
\item The {\it fine topology} on $E$ (associated with $\mathcal{U}$) is the coarsest topology on $E$ 
such that every $\mathcal{U}_q$-excessive function is continuous for some (hence all) $q>0$.
\item A topology $\tau$ on $E$ is called natural if it is a Lusin topology (i.e. $(E,\tau)$ is homeomorphic to a Borel subset of a compact metrizable space) which is coarser than the fine topology, and whose Borel $\sigma$-algebra is $\mathcal{B}$.
\end{enumerate}
\end{defi}
\begin{rem} \label{rem 4.3}
The necessity of considering natural topologies comes from the fact that, in general, the fine topology is neither metrizable, nor countably generated; see also Corollary \ref{coro 4.10} below.
\end{rem}

There is a convenient class of natural topologies to work with (as we do in Section 2), especially when the aim is to construct a right process associated to $\mathcal{U}$ (see Definition \ref{defi 4.4}). These topologies are called Ray topologies, and are defined as follows.

\begin{defi} \label{defi 4.5}
\begin{enumerate}[(i)]
\item If $\beta >0$ then a Ray cone associated with $\mathcal{U}_\beta$ is a cone $\mathcal{R}$ of bounded $\mathcal{U}_\beta$-excessive functions which is separable in the supremum norm, min-stable, contains the constant function $1$, generates $\mathcal{B}$, and such that $U_\alpha((\mathcal{R}-\mathcal{R})_+) \subset \mathcal{R}$ for all $\alpha > 0$.
\item A {\it Ray topology} on $E$ is a topology generated by a Ray cone.
\end{enumerate} 
\end{defi}

\begin{rem} \label{rem 4.6}
\begin{enumerate}[(i)]
\item Clearly, any Ray topology is a natural topology.
\item By e.g. \cite{BeBo04a}, Proposition 1.5.1, \cite{BeBoRo06} for the non-transient case, or \cite{BeRo11a}, Proposition 2.2, a Ray cone always exists and may be constructed as follows: start with a countable subset $\mathcal{A}_0\subset p\mathcal{B}$ which separates the points of $E$, and define inductively 
\begin{align*}
&\mathcal{R}_0:=U_\beta (\mathcal{A}_0)\cup \mathbb{Q}_+\\
&\mathcal{R}_{n+1}:=\mathbb{Q}_+\cdot \mathcal{R}_n \cup (\mathop{\sum}\limits_f\mathcal{R}_n )\cup (\mathop{\bigwedge}\limits_f \mathcal{R}_n) \cup (\mathop{\bigcup}\limits_{\alpha \in \mathbb{Q}_+}U_\alpha(\mathcal{R}_n))\cup U_\beta((\mathcal{R}_n-\mathcal{R}_n)_+),
\end{align*}
where by $\mathop{\sum}\limits_f \mathcal{R}_n$ resp. $\mathop{\bigwedge}\limits_f\mathcal{R}_n$ we denote the space of all finite sums (resp. infima) of elements from $\mathcal{R}_n$.
Then, a Ray cone $\mathcal{R}$ is obtained by taking the closure of $\bigcup\limits_n \mathcal{R}_n$ w.r.t. the supremum norm.
\end{enumerate}
\end{rem}

\noindent
{\bf Right processes.} 
Let now $X=(\Omega, \mathcal{F}, \mathcal{F}_t , X(t), \theta(t) , \mathbb{P}^x)$ be a normal Markov process with state space $E$ and shift operators $\theta(t):\Omega\rightarrow \Omega, \; t\geq 0$. 
In spite of the applications considered in this paper, but also in order to avoid the lifetime formalism, 
we assume that $X$ has infinite lifetime in $E$; otherwise, there is a standard way of adding a cemetery point $\Delta$ to $E$, and develop the forthcoming theory on $E\cup\{\Delta\}$; see \cite{BlGe68}, \cite{Sh88}, or \cite{BeBo04a}. 

 We assume that $X$ has resolvent $\mathcal{U}$ fixed above, i.e. for all $f\in b\mathcal{B}$ and $\alpha >0$
$$
U_\alpha f(x)=\mathbb{E}^{x}\Big\{\int_0^{\infty}e^{-\alpha t}f(X(t))dt\Big\},\quad  x\in E.
$$

To each probability measure $\mu$ on $(E, \mathcal{B})$ we associate the probability $\mathbb{P}^\mu (A):=\mathop{\int} \mathbb{P}^x(A)\; \mu(dx)$ for all $A \in \mathcal{F}$, and we consider the following enlarged filtration
$$
\widetilde{\mathcal{F}}_t:= \bigcap\limits_\mu \mathcal{F}_t^\mu, \; \; \widetilde{\mathcal{F}}:= \bigcap\limits_\mu \mathcal{F}^\mu,
$$
where $\mathcal{F}^\mu$ is the completion of $\mathcal{F}$ under $\mathbb{P}^\mu$, and $\mathcal{F}_t^\mu$ is the completion of $\mathcal{F}_t$ in $\mathcal{F}^\mu$ w.r.t. $\mathbb{P}^\mu$; in particular, $(x,A)\mapsto\mathbb{P}^{x}(A)$ is assumed to be a kernel from $(E,\mathcal{B}^{u})$ to $(\Omega, \mathcal{F})$, where $\mathcal{B}^{u}$ denotes the $\sigma$-algebra of all universally measurable subsets of $E$.

\begin{defi} \label{defi 4.4}
The Markov process $X$ is called a right (Markov) process if the following additional hypotheses are satisfied:
\begin{enumerate}[(i)]
\item The filtration $(\mathcal{F}_t)_{t\geq 0}$ is right continuous and $\mathcal{F}_t=\widetilde{\mathcal{F}}_t, t\geq 0$.
\item For one (hence all) $\alpha>0$ and for each $f \in \mathcal{E}(\mathcal{U}_\alpha)$ the process $f(X)$ has right continuous paths $\mathbb{P}^{x}$-a.s. for all $x\in E$.
\item There exists a natural topology on $E$ with respect to which the paths of $X$ are $\mathbb{P}^{x}$-a.s. right continuous for all $x\in E$.
\end{enumerate}
\end{defi}

\begin{rem} 
\begin{enumerate}[(i)]
\item By \cite{Sh88}, Theorem 7.4, condition (ii) of the previous definition is sufficient to be checked on functions $f$ of the form $U_\alpha g$, for all $g$ bounded and continuous (w.r.t. a natural topology). 
\item We emphasize that Definition \ref{defi 4.4} uses no a priori topology on $E$, which is endowed merely with the $\sigma$-algebra $\mathcal{B}$. 
It is the resolvent $\mathcal{U}$ which provides the topology, namely the fine topology; in spite of Remark \ref{rem 4.3} also, this is why it is of interest to consider Markov processes which have path-continuity properties w.r.t. natural topologies; see also Remark \ref{rem 4.11} below.
\end{enumerate}
\end{rem}

According to \cite{BlGe68}, Chapter II, Theorem 4.8, or \cite{Sh88}, Proposition 10.8 and Exercise 10.18, Definition \ref{defi 4.4} leads to a key probabilistic understanding of the fine topology, namely:

\begin{thm} \label{thm 4.6} If $X$ is a right process, then a universally $\mathcal{B}$-measurable function $f$ is finely continuous if and only if $(f(X(t)))_{t\geq 0}$ has $\mathbb{P}^{x}$-a.s. right continuous paths for all $x\in E$.
\end{thm}

\begin{coro} \label{coro 4.10} If $X$ is a right process, then it has a.s. right continuous paths w.r.t. any natural topology on $E$.
\end{coro}
\begin{proof}
Since a natural topology is countably generated, i.e. it is the initial topology of a countable family of functions, the statement follows by Thereom \ref{thm 4.6}.
\end{proof}

\begin{rem} \label{rem 4.11}
When we are given a certain topology $\tau$ on $E$ whose Borel $\sigma$-algebra is $\mathcal{B}$, it is a separate issue to determine if $\tau$ is natural, in which case, any right process would have a.s. right continuous paths w.r.t. $\tau$. We deal with this problem, in general and concrete situations, in the main body of this paper.
\end{rem}

In view of Definition \ref{defi 2.1}, (iii), let us recall what a polar set is, from both analytic and probabilistic perspectives.

\begin{defi}
\begin{enumerate}[(i)]
\item If $u\in \mathcal{E}(\mathcal{U}_\alpha)$ and $A \in \mathcal{B}$, then the $\alpha$-order reduced function of $u$ on $A$ is given by
$$
R_\alpha^A u = \inf \{ v \in \mathcal{E}(\mathcal{U}_\alpha): \, v\geqslant u \mbox{ on } A \}.
$$
$R_\alpha^A u$ is merely supermedian w.r.t $\mathcal{U}_\alpha$, and we denote by $B_\alpha^A u=\widehat{R_\alpha^A u}$ it's excessive regularization, called the balayage of $u$ on $A$.
\item A set $A\in \mathcal{B}$ is called polar if $B_\alpha^A 1=0$.
\end{enumerate}
\end{defi}

The following fundamental identification due to G.A. Hunt holds (see e.g. \cite{DeMe78}):
\begin{thm} \label{thm 4.13}
If $X$ is a right process, then for all $u\in \mathcal{E}(\mathcal{U}_\alpha)$ and $A\in\mathcal{B}$
$$
B_\alpha^A u=\mathbb{E}^x\{e^{-\alpha T_A} u (X(T_A))\},
$$
where $T_A:=\inf \{ t>0 :  X(t)\in A \}$.
In particular, $A$ is polar if and only if $\mathbb{P}^{x}(T_A<\infty)=0$ for all $x\in E$.
\end{thm}

Without further conditions, the assumption (H') from the beginning of this section, although necessary, is not sufficient to ensure the existence of a right process associated with $\mathcal{U}$ (see Corollary  \ref{rem 3.16}).
Nevertheless, there is always a larger space on which such a process exists, and let us briefly recall it's construction.

We denote by $Exc(\mathcal{U}_\beta)$ the set of all $\mathcal{U}_\beta$-excessive measures 
($\xi \in Exc(\mathcal{U}_\beta)$ if and only if $\xi$ is a $\sigma$-finite measure on $M$ and $\xi \circ \alpha U_{\alpha+\beta} \leq \xi$ for all $\alpha >0$).

\begin{defi} Let $\beta >0$.
\begin{enumerate}[(i)]
\item $Exc(\mathcal{U}_\beta)$ denotes the set of all $\mathcal{U}_\beta$-excessive measures: 
$\xi \in Exc(\mathcal{U}_\beta)$ if and only if $\xi$ is a $\sigma$-finite measure on $M$ and $\xi \circ \alpha U_{\alpha+\beta} \leq \xi$ for all $\alpha >0$.
\item The {\it energy functional} associated with $\mathcal{U}_\beta$ is $L^{\beta}: Exc(\mathcal{U}_\beta)\times \mathcal{E}(\mathcal{U}_\beta) \rightarrow \overline{\mathbb{R}}_+$ given by
$$
L^{\beta}(\xi,v):=\sup\{\mu(v) \; : \; \mu \mbox{ is a } \sigma\mbox{- finite measure, } \mu \circ U_\beta \leq \xi\}
$$
\item The {\it saturation} of $E$ (with respect to $\mathcal{U}_\beta$) is the set $E_1$ of all extreme points of the set $\{\xi \in Exc(\mathcal{U}_\beta)\; : \; L^{\beta}(\xi,1)=1\}$.
\end{enumerate}
\end{defi}

The map $E\ni x \mapsto \delta_x \circ U_\beta \in Exc(\mathcal{U}_\beta)$ is an embedding of $E$ into $E_1$ and 
every $\mathcal{U}_\beta$-excessive function $v$ has an extension $\widetilde{v}$ to $E_1$, defined as $\widetilde{v}(\xi):=L^{\beta}(\xi,v)$.
The set $E_1$ is endowed with the $\sigma$-algebra $\mathcal{B}_1$ generated by the family $\{\widetilde{v}: \; v\in \mathcal{E}(\mathcal{U}_\beta)\}$.
In addition, as in \cite{BeBoRo06}, sections 1.1 and 1.2., there exists a unique resolvent of kernels $\mathcal{U}^{1}=(U^{1}_\alpha)_{\alpha>0}$ on $(E_1, \mathcal{B}_1)$ 
which is an extension of $\mathcal{U}$ in the sense of Definition \ref{defi 2.2}, and it satisfies the assumption (H') from the beginning of this section; more precisely, it is given by
\begin{equation} \label{eq 4.1}
U^{1}_\alpha f(\xi)=L^{\beta}(\xi, U_\alpha (f|_M)) 
\mbox{ for all } f \in bp\mathcal{B}_1, x\in M_1, \alpha >0.
\end{equation}
Note that $(E_1,\mathcal{B}_1)$ is a Lusin measurable space, the map $x \mapsto\delta_x \circ U_\beta$ identifies $E$ with a subset of $E_1$, $E\in \mathcal{B}_1$ and $\mathcal{B}=\mathcal{B}_1|_E$.

We end this section with the following key result 
on which our work from Section 2 heavily relies on
 (cf.  (2.3)  from \cite{BeRo11a},  sections 1.7 and 1.8 in \cite{BeBo04a},  Theorem 1.3 from \cite{BeBoRo06}, and section  3 in \cite{BeBoRo06a}):

\begin{thm} \label{thm 4.15}
There is always a right process on the saturation $(E_1,\mathcal{B}_1)$, associated with $\mathcal{U}^{1}$. Moreover, the following assertions are equivalent:
\begin{enumerate}[(i)]
\item There exists a right process on $E$ associated with $\mathcal{U}$.
\item The set $E_1\setminus E$ is polar (w.r.t. $\mathcal{U}_1$).
\end{enumerate}
\end{thm}

We give now a simple example of a strong Feller transition function on a one dimensional space, whose resolvent satisfies hypothesis (H') but does not have  associated a right process.

\begin{coro} \label{rem 3.16}
Let $(P_t)_{t\geq 0}$ be the $1$-dimensional Gaussian semigroup, 
i.e., the transition function of the real valued Brownian motion.
Let  $E:=\mathbb{R}\setminus \{0\}$ and consider  the restriction $\mathcal{V}$ of 
$\mathcal{U}$ from $\mathbb{R}$ to $E$ (which is possible because  $U_\alpha (1_{\{0\}})\equiv 0, \alpha >0$). 
Then $\mathcal{V}$ is  $C^{\infty}_b$-strong Feller (i.e. each kernel $V_\alpha$ maps bounded functions to $C^{\infty}_b$-functions), it 
satisfies (H'),  and there is no right process on $E$ with resolvent $\mathcal{V}$.
\end{coro}

\begin{proof}
Clearly, the semigroup $(P_t)_{t\geq 0}$ is $C^{\infty}_b$-strong Feller.
Consequently, so is its resolvent $\mathcal{U}$, which in addition satisfies (H') and $\mathcal{V}$ preserves the same properties.
Further, we can argue as in \cite{BeRo11a}, the Example at page 849.
One can easily see that  $Exc(\mathcal{U}_1)=Exc(\mathcal{V}_1)$ and that $\delta_0\circ U_1$ is extremal in $Exc(\mathcal{U}_1)$.  
Therefore, if $E_1$ is the saturation of $E$ then we have the embeddings $E\subset \mathbb{R} \subset
E_1$ and $\mathcal{U}=\mathcal{U}^{1}|_{\mathbb{R}}$. 
Now, if $\mathcal{V}$ admits a right process on $E$, by Theorem \ref{thm 4.15} we have that the set $\{0\}$ is polar w.r.t. $\mathcal{U}$. But by Theorem \ref{thm 4.13} it means that the real valued Brownian motion never hits $0$, which is of course wrong.
\end{proof}

\vspace{0.3cm}
\noindent
{\bf Acknowledgments.} 
 Financial support by the Deutsche Forschungsgemeinschaft, project number  CRC 1283 is gratefully acknowledged.
 For the  first and the second named authors this work was supported by a grant of Ministry of Research and Innovation, CNCS - UEFISCDI, project number PN-III-P4-IDPCE-2016-0372, within PNCDI III.

\end{document}